\documentclass[pdflatex,sn-mathphys-num]{sn-jnl}% Math and Physical Sciences Numbered Reference Style
\usepackage{graphicx}%
\usepackage{multirow}%
\usepackage{amsmath,amssymb,amsfonts}%
\usepackage{amsthm}%
\usepackage{mathrsfs}%
\usepackage[title]{appendix}%
\usepackage[dvipsnames]{xcolor}%
\usepackage{textcomp}%
\usepackage{manyfoot}%
\usepackage{booktabs}%
\usepackage{algorithm}%
\usepackage{algorithmicx}%
\usepackage{algpseudocode}%
\usepackage{listings}%
\usepackage{physics}
\newcommand{\R}{\mathbb R}

\newcommand{\mcM}{\mathcal M}

\newcommand{\mcU}{\mathcal U}

\newcommand{\mcP}{\mathcal P}
\newcommand{\mcS}{\mathcal S}

\newcommand{\colspan}{\operatorname{span}}
\newcommand{\ran}{\operatorname{ran}} 
\newcommand{\sym}{\operatorname{sym}}
\newcommand{\Skew}{\operatorname{skew}} %\skew exists already in LaTeX
\newcommand{\dist}{\operatorname{dist}}
\newcommand{\Exp}{\operatorname{Exp}}
\newcommand{\Log}{\operatorname{Log}}

\newcommand{\Gr}{\operatorname{Gr}}

\newcommand{\cond}{\operatorname{cond}}
\DeclareMathOperator{\T}{\textnormal{T}}

\newcommand{\St}{\operatorname{St}}

\usepackage{mathtools}

\DeclarePairedDelimiterX{\inp}[2]{\langle}{\rangle}{#1, #2}
\DeclarePairedDelimiterX{\inpF}[2]{\langle}{\rangle_{\textnormal{F}}}{#1, #2}

\newcommand{\deriv}[1][t]{\frac{\mathrm{d}}{\mathrm{d}#1}}

\newcommand{\diff}[1][t]{\mathrm{d}}
\newcommand{\Diff}[1][t]{\mathrm{D}}

\makeatletter
\renewcommand*\env@matrix[1][*\c@MaxMatrixCols c]{%
	\hskip -\arraycolsep
	\let\@ifnextchar\new@ifnextchar
	\array{#1}}
\makeatother

\newcount\Comments  % 0 suppresses notes to selves in text
\newcommand{\kibitz}[2]{\ifnum\Comments=1\textcolor{#1}{#2}\fi}

\newcommand{\normF}[1]{\norm{#1}_{\textnormal{F}}}

\usepackage{amsopn}
\DeclareMathOperator{\diag}{diag}
\usepackage[outline]{contour}
\contourlength{1.5pt}
\usepackage{tikz}
\usepackage{pgfplots}
\usepgfplotslibrary{fillbetween}
\usetikzlibrary{calc}
\usetikzlibrary{plotmarks}
\usetikzlibrary{patterns}

\usepackage{enumitem}

\usepackage{tikzducks}
\usepackage{cancel}

\usepackage[normalem]{ulem}
\usepackage{wasysym}

\usepackage{cleveref}

\theoremstyle{thmstyleone}
\newtheorem{lemma}{Lemma}

\theoremstyle{thmstyleone}%
\newtheorem{theorem}{Theorem}%  meant for continuous numbers
\theoremstyle{thmstyletwo}%
\newtheorem{remark}{Remark}%

\theoremstyle{thmstylethree}%

\begin{document}

\title[Article Title]{Maximum volume coordinates for Grassmann interpolation: Lagrange, Hermite, and errors}


\author*[1]{\fnm{Rasmus} \sur{Jensen}}\email{rasmusj@imada.sdu.dk}

\author*[1]{\fnm{Ralf} \sur{Zimmermann}}\email{zimmermann@imada.sdu.dk}

\affil[1]{\orgdiv{Department of Mathematics and Computer Science}, \orgname{University of Southern Denmark}, \orgaddress{\street{Campusvej 55}, \city{Odense}, \postcode{5230}, \country{Denmark}}}


\abstract{    We present a novel approach to Riemannian interpolation on the Grassmann manifold. Instead of relying on the Riemannian normal coordinates, i.e. the Riemannian exponential and logarithm maps, we approach the interpolation problem with an alternative set of local coordinates and corresponding parameterizations. 
    A special property of these coordinates is that their calculation does not require any matrix decompositions.
    This is a numerical advantage over Riemann normal coordinates and many other retractions on the Grassmann manifold, especially when derivative data are to be treated.

    To estimate the interpolation error, we examine the conditioning of these mappings and state explicit bounds. It turns out that the parameterizations are well-conditioned, but the coordinate mappings are generally not.
    As a remedy, we introduce maximum-volume coordinates that are based on a search for subblocks of column-orthogonal matrices of large absolute determinant. We show that the order of magnitude of the asymptotic interpolation error on $\Gr(n,p)$  is the same as in the Euclidean space. 
    
    Two numerical experiments are conducted. The first is an academic one, where we interpolate a parametric orthogonal projector $QQ^T$, where the $Q$--factor stems from a parametric compact QR--decomposition. The second experiment is in the context of parametric model reduction of dynamical systems, where we interpolate reduced subspaces that are obtained by proper orthogonal decomposition.}



\pacs[MSC Classification]{15B10, %Orthogonal matrices
  15B57, %Hermitian, skew-Hermitian, and related matrices
  65F99, % Numerical linear algebra, None of the above, but in this section
  53C30, %Differential geometry of homogeneous manifolds
  53C80, %Applications of global differential geometry to the sciences
  }

\maketitle

\section{Introduction}
Numerical operations on nonlinear Riemannian manifolds are commonly referred to as `Riemannian computing'.
Riemannian computing methods have established themselves as important tools in a large variety of applications, including  computer vision, machine learning, and optimization, see
\cite{AbsilMahonySepulchre2004,Absilbook:2008,EdelmanAriasSmith:1999,Gallivan_etal2003, Lui2012, sato2021, boumal2023}, and the anthologies \cite{Minh:2016:AAR:3029338, RiemannInComputerVision}.
They also have applications in statistics and data science
\cite{Fletcher:2020}
and in numerical methods for differential equations \cite{BennerGugercinWillcox2015,Celledoni_2020,hairer06gni, iserles_munthe-kaas_norsett_zanna_2000}. 

A special subclass of Riemannian computing problems is the interpolation of manifold-valued functions. 
Suppose that 
\[
    f: D \to \mcM, \quad x \mapsto f(x)
\]
is a differentiable function on a Euclidean domain $D\subset \R^d$ with outputs on a Riemannian manifold $(\mcM, \langle\cdot,\cdot\rangle)$.
Given a sample data set $\{(x_i, f(x_i)) | i=1,\ldots,k\}$,
manifold interpolation is concerned with computing an {\em interpolant}
\[
  \hat f: D \to \mcM, \quad x \mapsto \hat f(x) \quad \text{ with } \hat f(x_i) = f(x_i), \quad i=1,\ldots,k. 
\]
The basic interpolation task may be enhanced by including samples of derivative data of $f$, which then goes by the name of Hermite interpolation.
Classical Euclidean interpolation algorithms (Lagrange Newton, Hermite, radial basis functions, ...) rely heavily on the underlying vector space structure; a feature which manifolds are lacking. Therefore, the challenge in manifold interpolation is to ensure that the interpolant $\hat f$ respects the manifold structure.\\
The standard approach to manifold interpolation can be subdivided in three key steps:\footnote{This approach works, if the sampled data can be covered by a single coordinate chart. Otherwise, extra measures have to be taken to ensure a smooth transition of the interpolant between overlapping coordinate domains.}
\begin{itemize}[align=left]
    \item[Step 1:] {\bf Preprocessing.} Map the sample data from the manifold $\mcM$ to a (Euclidean) coordinate domain. This requires one to select a suitable coordinate chart.
    \item[Step 2:] {\bf Interpolation.}
    Employ a Euclidean interpolation method that is linear in the sample data to interpolate the data images in the coordinate domain.
    \item[Step 3:] {\bf Postprocessing.} Map the interpolated data from the Euclidean coordinate domain to the manifold.
    For consistency, this has to be done with the inverse of the coordinate chart selected in step 1.
\end{itemize}
In this work, we investigate interpolation methods on the Grassmann manifold $\Gr(n,p)$ of $p$-dimensional linear subspaces  of $\R^n$. 
Interpolation of subspaces can be used for {\em parametric model reduction},
where Grassmann data points appear as subspaces of dominant modes, usually associated with a complex dynamical system, \cite{AmsallemFarhat2008, Zimmermann2014, BennerGugercinWillcox2015, GentleGrassmannInt2024, HessQuainiRozza:2023, Zim21, ElOmari:2025, Pfaller:2020}.

This is also our main motivation to investigate Grassmann interpolation problems. The natural habitat of model reduction problems is high-dimensional spaces. Therefore, computational efficiency is of special concern and we work under the informal `low-rank' assumption that $n\gg p$, i.e., the dimension of the subspaces is much lower than that of the ambient space $\R^n$.

\paragraph{Original contribution}
Two issues are of special interest in the context of manifold interpolation. On the one hand, the numerical feasibility and efficiency, on the other hand estimates and control of the interpolation error. With the work at hand, we aim to contribute to both of them.

With regard to the first point, we examine local Grassmann coordinates that do not require the computation of matrix decompositions.
Although known for at least three decades, see \cite[Appendix C.4]{HelmkeMoore1994},
to the best of our knowledge, these coordinates have not yet been used for interpolation applications, nor have they been studied theoretically. 
We show that this set of coordinates gives rise to a retraction on the Grassmann manifold. Moreover, we quantify the conditioning of the coordinates and their inverses, which are the associated parameterizations. Further, we discuss
how to optimize the coordinates with respect to their condition number, and why this is important.
For reasons that will become clear later, we refer to the optimized coordinates as \textit{maximum--volume coordinates} (MV coordinates). 

The property of being `matrix decomposition--free' should be seen in contrast to the Riemannian normal coordinates. In a certain geometric sense, Riemannian normal coordinates (see \Cref{app:rie_norm_coords}) are the best possible option for interpolation problems. However, their numerical realizations on the Grassmann manifold rely on the singular value decomposition and matrix exponentials and/or matrix sine and cosine functions. The advantage of matrix decomposition--free coordinates is most prominent when Hermite interpolation is performed, because processing derivative data requires differentiating matrix decompositions when Riemannian normal coordinates are employed \cite{ZimmermannHermite:2020}.

With regard to the second point, we investigate sources of error in all of the aforementioned three steps `preprocessing', `interpolation', `postprocessing'. In particular, we examine how interpolation errors that stem from the Euclidean interpolation scheme are forwarded from the coordinate domain to the manifold.

Finally, we demonstrate the findings by means of numerical experiments, where Lagrange- and Hermite interpolation on the Grassmann manifold are considered. We observe that interpolation in MV coordinates leads to interpolants that exhibit an accuracy comparable to their Riemannian normal coordinate counterparts. 

\paragraph{State of the art and related work}

Univariate Hermite manifold interpolation has been considered explicitly in \cite{Jakubiak:2006,ZimmermannHermite:2020, SeguinKressner:2024}.
A related line of research is the generalization of Bézier curves and the De Casteljau-algorithm
to Riemannian manifolds, see \cite{BergmannGousenbourger:2018,GouseMassartAbsil:2018, Polthier:2013,  Noakes:2007, SAMIR:2019}.
Subdivision schemes based on geodesic averages have been investigated in \cite{Dyn:2017, Wallner:2005,BenZion:2022}.
A survey on manifold interpolation methods for applications in model reduction is given in \cite{Zim21}.
Recently, a very accessable tutorial focusing on Grassmann interpolation has appeared \cite{GentleGrassmannInt2024}.

We are aware of four results on manifold interpolation errors:
\begin{enumerate}
    \item The appendix of \cite{Attali2007} contains sandwich bounds `Euclidean distance vs. manifold distance' for embedded submanifolds where the Riemannian metric is the one obtained from the ambient Euclidean space.
    While not of concern in \cite{Attali2007}, these bounds can be used to estimate manifold interpolation errors. The recent preprint \cite{mataigne2024bounds} gives specific comparative bounds for the Stiefel manifold.
    \item The work \cite{ZimmermannHermite:2020} includes a result on error propagation `tangent space to manifold' that relies on Riemannian normal coordinates and the local sectional curvature.
    A consequence of this result is that if the tangent space interpolation scheme features an asymptotic error of $\mathcal{O}(h^k)$ for interpolation steps $h\to 0$, then so does the manifold interpolation scheme.
    \item By covariant calculus, it is shown in \cite{SeguinKressner:2024} that Hermite interpolation using retractions and a Riemannian version of the De Casteljau algorithm
    exhibits an asymptotic error of $\mathcal{O}(h^4)$, $h$ being the interpolation step size, just like with the Euclidean counterpart.
    \item Recently in \cite{Jacobsson:2024}, an error bound associated to approximating functions on Riemannian manifolds is established by applying Riemannian comparison theory.
    The error bound obtained is similar to the one from \cite{ZimmermannHermite:2020}, but relies on the Toponogov comparison theorem from Riemannian geometry and thus requires global curvature bounds rather than local ones. This result also shows that the order of the asymptotic interpolation error is preserved when transiting from the tangent space to the manifold. 
\end{enumerate}
To the best of our knowledge, the numerical errors associated with the preprocessing step of Riemannian interpolation, i.e, with the mapping of the sample data onto a suitable coordinate domain have not yet been considered in the literature.

\paragraph{Organisation of the paper}
In \Cref{sec:background_geo} we recall some background theory on the Grassmann manifold and a few facts from general Riemannian geometry. In \Cref{sec:MaxVolCoords} we discuss interpolation in local and maximum-volume coordinates, and consider data processing and interpolation errors. We apply the interpolation schemes in two numerical experiments in \Cref{sec:experiments} and summarize our findings in \Cref{sec:conclusions}.

\paragraph{Notation}
For the reader's convenience, we list the main symbols and variables.

\begin{small}
\begin{tabular}{@{}ll@{}}
%\toprule
Symbol &  meaning \\
\midrule
$I,I_n$   & identity matrix, provided with a dimensional index if required \\
$I_{n,p}$ & Rectangular `identity' $I_{n,p}= \begin{pmatrix}
                I_p\\
                0\end{pmatrix}\in \R^{n\times p}$\\
$\inpF{\cdot}{\cdot}, \normF{\cdot}$ & Euclidean metric $\inpF{X}{Y} = \tr(X^T Y)$ with Frobenius norm\\
$\langle \cdot,\cdot\rangle_0,\hspace{0.1cm} \|\cdot\|_0$ & canonical metric on Grassmann $\langle X, Y\rangle_0 = \frac12\tr(X^T Y)$ and norm \\
$O(n)$    & orthogonal group $O(n) = \{Q\in \R^{n\times n}\mid Q^TQ = I_n\}$\\
$\sym(n)$ & vector space of symmetric matrices $\{A\in \R^{n\times n}\mid A^T = A\}$.\\
$\Skew(n)$ & vector space of skew-symmetric matrices $\{A\in \R^{n\times n}\mid A^T = -A\}$.\\
$\mathcal{M}$ & a Riemannian manifold\\
$\St(n,p)$ & Stiefel manifold $\St(n,p)=\{U\in\R^{n\times p}\mid U^TU = I_p\}$\\
%$\T_U\St(n,p)$ & Stiefel tangent space at $U$, $T_U\St(n,p)=\{\Delta \mid U^T\Delta +\Delta^TU =0\}$\\
$\Gr(n,p)$ & Grassmann mnf. $\Gr(n,p)=\{P\in \sym(n): P^T=P, \rank(P)=p\}$\\
%$\T_U\Gr(n,p)$ & Grassmann tangent space at $U$, $T_U\Gr(n,p)=\{\Delta \mid U^T\Delta =0\}$\\
$\colspan(X)$ & span of a set of vectors or range of a matrix $X$ (context dependent)\\
$\cond_{(\cdot)}(f)$ & absolute condition number of $f$ wrt. the norm indicated by $(\cdot)$\\
$\exp_m$ & matrix exponential $\exp_m(X) = \sum_{k=0}^\infty \frac{1}{k!}X^k$\\
$\hat f$& the interpolant of a function $f$\\
$\qty{(t_i,f_i)}_{i=1}^d$& (Lagrange) data set of sample points plus sample values\\
$\qty{(t_i,f_i,f_i')}_{i=1}^d$ & (Hermite) data set of sample points, sample values and derivatives
\end{tabular}
\end{small}

\section{Background}\label{sec:background_geo}
This section covers the necessary geometrical background and numerical aspects of the Grassmann manifold.
Thanks to works such as \cite{EdelmanAriasSmith:1999, AbsilMahonySepulchre2004} and textbooks \cite{Absilbook:2008, boumal2023, sato2021, gallier2011geometric}, the Grassmann manifold is familiar to an increasing number of matrix analysts, so we try to keep our review short.
For a detailed survey see, \cite{Bendokat:2024}. 

\subsection{The Grassmann manifold}
\label{sec:Grassmann}

The Grassmann manifold (or Grassmannian) $\Gr(n,p)$ is the set of all $p$ dimensional subspaces of $\R^n$,
\[
    \Gr(n,p) = \{\mathcal{U}\leq \R^n\mid \dim\mathcal{U} = p\}.\footnote{When used for subsets of vector spaces, the symbol `$X\leq Y$' means '$X$ is a subspace of $Y$'.}
\]
A subspace $\mathcal{U}$ may be uniquely represented by
\begin{enumerate}
\item $\mathcal{U} = \colspan(U)$, where $U\in \R^{n\times p}$ is a matrix whose columns form a basis of $\mathcal{U}$.
Mind that $\colspan(U)$ is unique, but the spanning matrix $U$ is not.
\item The orthogonal projector $P_\mathcal{U}:\R^n \to \R^n$ onto $\mcU$.
This is the unique linear, idempotent map ($P_\mathcal{U}\circ P_\mathcal{U}=P_\mathcal{U}$) with $\ran(P_\mathcal{U})= \mathcal{U}$, $\ker(P_\mathcal{U})= \mathcal{U}^{\bot}$.
The orthogonal complement is with respect to the Euclidean inner product.
Orthogonal projectors are equivalently characterized by being self-adjoint, $P_{\mcU}^T = P_{\mcU}$, or by featuring unit norm $\|P_{\mcU}\|_2=1$, which is the smallest possible value a projector can exhibit, see \cite[Section 3.1]{Eigenvalue_Problems_Saad}.
\end{enumerate}
For the numerical representation of subspaces, we will use column-orthogonal matrices exclusively.
The set of all such matrices, i.e., the set of orthonormal $p$--frames is called the Stiefel manifold $$\St(n,p)=\{U\in \R^{n\times p}\mid U^TU=I_p\}.$$
Each subspace $\mcU\in \Gr(n,p)$ corresponds to an equivalence class of Stiefel matrices $[U]=\{UR\mid R\in O(p)\}\subset \St(n,p)$.
Any $U\in [U]$ is called a matrix representative of the subspace $[U]$.
When formalized, this construction shows that the Grassmann manifold is a quotient manifold of the Stiefel manifold
\[
 \Gr(n,p) = \St(n,p)/O(p),
\]
see 
\cite{EdelmanAriasSmith:1999, Absilbook:2008, Bendokat:2024, Zim21}, and \cite[Chapter 21]{LeeSmooth:2012}  for the full theoretical background.
Given any not necessarily column-orthonormal $U\in \R^{n\times p}$ with $\colspan(U)=\mcU$, 
the orhogonal projector onto $\mcU$ is
\[
    P_{\mcU} = U(U^TU)^{-1}U^T.
\]
For 
$U\in [U] \subset \St(n,p)$, the formula reduces to $P_{\mcU} = UU^T$.
This matrix representation is unique and independent from the chosen matrix representative $U\in [U]\subset \St(n,p)$.

Both the Stiefel and the Grassmann manifold are smooth manifolds; yet the former is embedded in the Euclidean matrix space $\R^{n\times p}$, while the latter is abstract when viewed as a quotient space of equivalence classes.
However, by considering the Grassmann manifold as the set of orthogonal projectors
\[
    \Gr(n,p) = \{P\in \R^{n\times n}\mid P^2 =P, P^T=P,  \rank(P) = p\}\subset \sym(n),
\]
it is an embedded submanifold of $\sym(n)\subset \R^{n\times n}$.

The archetype projector onto the first $p$ Cartesian coordinates is
\begin{equation}
\label{eq:P0}
    P_0 = 
    \begin{pmatrix}
        I_p & 0\\
        0   & 0
    \end{pmatrix}
    =
    \begin{pmatrix}
        I_p\\
        0
    \end{pmatrix}\begin{pmatrix}
        I_p& 0
    \end{pmatrix}
    =: I_{n,p} I_{n,p}^T.
\end{equation}
An eigenvalue decomposition shows that any other projector is obtained from an orthogonal similarity transformation on $P_0$, 
$$
P\in \Gr(n,p)
	\Leftrightarrow \exists Q\in O(n): P = QP_0Q^T.
$$ 
Note that if we split $Q=\begin{pmatrix} U & U_\bot\end{pmatrix}\in O(p)$ into the first $p$ columns and the tailing $(n-p)$ columns, then 
$P = QP_0Q^T = UU^T$.\\
The similarity transformation
is associated with an isometric group action 
\begin{equation}
\label{eq:projector_group_action}
		\Phi: O(n)\times \R^{n\times n} \to \R^{n\times n},
		(Q, M) \mapsto QMQ^T,
\end{equation}
and the Grassmann set of orthogonal projectors is the  {\em orbit} of $P_0$ under this action.

The dimensions of the manifolds $\Gr(n,p)$ and $\St(n,p)$ are $np-p^2$ and $np-\frac{1}{2}p(p+1)$ respectively. 

\paragraph{Tangent spaces and Riemannian metric}
The tangent space associated with a manifold location is the vector space of velocity vectors of differentiable curves passing through the chosen location. For the Grassmann manifold as the quotient space $\Gr(n,p) = \St(n,p)/O(p)$,
the matrix representation of the tangent space is 
\begin{equation*}
  \T_{\mcU}\Gr(n,p)= \left\{D \in \R^{n\times p}  ~\mid~ U^T D = 0\right\}\subset \R^{n\times p}.
\end{equation*}
From the projector perspective,
the tangent space at $P=QP_0Q^T$ is
\begin{equation}
\label{eq:GrTan}
    \T_P\Gr(n,p) = \left\{
    Q\begin{pmatrix}
        0&\Delta^T\\
        \Delta&0
    \end{pmatrix}Q^T~\mid~ \Delta\in \R^{(n-p)\times p}
    \right\} = Q \qty(\T_{P_0}\Gr(n,p)) Q^T.
\end{equation}
Grassmann tangent vectors can be characterized in many equivalent ways;
from \cite[Prop. 2.1]{Bendokat:2024}, we record
\begin{equation}
\label{eq:GrTan_characterization}
    \forall X\in\sym(n): \quad X\in \T_P\Gr(n,p) \Leftrightarrow XP+PX=X.
\end{equation}

This implies that $2PXP=PXP$ so that $PXP=0$ for any Grassmann tangent vector $X\in \T_P\Gr(n,p)$.
The projector tangent space is a subset set of the Euclidean vector space $\sym(n)$. Given a tangent vector $X\in \T_P\Gr(n,p)$, we can compute the so--called \textit{horizontal lift}, and obtain a tangent vector in Stiefel representation.
Any Stiefel tangent vector $T\in T_U\St(n,p)$ is of the form
$T = UA  + U_\bot B$ with $A\in \Skew(p)$. According to the quotient construction, Grassmann tangent vectors may be identified with special Stiefel tangent vectors, namely horizontal ones.
Horizontal Stiefel tangent vectors are precisely those, where $A=0$.
A choice must be made as to where to lift. 
If $U\in\St(n,p)$ is a chosen Stiefel representative of $P=UU^T$, then the horizontal lift corresponding to $U$ is $X^{\textsf{hor}}=XU$.  

The Riemannian metric, i.e., the family of inner products on the collection of tangent spaces is obtained by restricting the inner product on $\sym(n)$ to the tangent spaces.
The metric on $\sym(n)$ and thus on any $\T_P\Gr(n,p)$ is
\begin{equation}
\label{eq:canonical_metric}
    \langle X, Y\rangle_0 = \frac12 \trace(X^TY).
\end{equation}
This is called the canonical metric on $\Gr(n,p)$.
The factor of $\frac{1}{2}$ can be motivated by not wanting to count the independent variables $\Delta$ in \eqref{eq:GrTan} twice. However, it is also required to comply with common conventions: As a rule, all geometric quantities on the Grassmannian that appear in the literature, such as angles, lengths, and curvature, are with respect to this metric \cite{Wong1967, Wong1968}.
It coincides with the canonical metric on the Grassmannian when considered as a quotient space \cite{EdelmanAriasSmith:1999}.
\paragraph{Geodesics}
The length of a curve $c:[a,b] \to \Gr(n,p)$ is $L(c) = \int_a^b \|\dot c(t)\|_0\dd t$, where the norm is the one induced by \eqref{eq:canonical_metric}.
Geodesics are critical candidates for being shortest connections on a manifold, and are uniquely determined by a starting point and an initial velocity.
Given a projector $P = QP_0Q^T$ and a tangent vector $X= Q\begin{pmatrix}
        0&\Delta^T\\
        \Delta&0
    \end{pmatrix}Q^T\in \T_P\Gr(n,p)$, the Grassmann geodesic that starts from $P$ with velocity $X$ is
\begin{align}
\nonumber
    c(t)=& \exp_m (t [X,P]) P \exp_m (-t[X,P])\\
\nonumber
    =&Q\exp_m \qty(t\begin{pmatrix}
        0&-\Delta^T\\
        \Delta&0
    \end{pmatrix})
    \begin{pmatrix}
        I_p\\
        0
    \end{pmatrix}\cdot \begin{pmatrix}
        I_p& 0
    \end{pmatrix}
    \exp_m \qty(-t\begin{pmatrix}
        0&-\Delta^T\\
        \Delta&0
    \end{pmatrix}) Q^T\\
    =:& \gamma(t)\cdot \gamma(t)^T.
    \label{eq:Grassmann_proj_geo}
\end{align}
This follows from \cite[Proposition 3.2]{Bendokat:2024}, because the commutator bracket yields\\ $[X,P] = XP-PX = Q\begin{pmatrix}
        0&-\Delta^T\\
        \Delta&0
    \end{pmatrix}Q^T$.
Observe that the Grassmann projector curve $c:t\to c(t)\in\Gr(n,p)$ splits into the `projector product' of the Stiefel curve $\gamma:t\to \gamma(t)\in \St(n,p)$.
This is the horizontal Stiefel geodesic starting from $\gamma(0) = QI_{n,p} = \begin{pmatrix}
    U & U_\bot
\end{pmatrix}I_{n,p} = U$ with horizontal velocity
$\dot\gamma(0) = U_\bot \Delta\in \T_U\St(n,p)$.
With the SVD $S\Sigma R^T = \Delta$, it 
holds 
$\begin{pmatrix}
    0&-\Delta^T\\
    \Delta&0
\end{pmatrix}=\begin{pmatrix}
        R& 0\\
        0& S
    \end{pmatrix}
    \begin{pmatrix}
        0&-\Sigma\\
        \Sigma&0
    \end{pmatrix}
    \begin{pmatrix}
        R^T& 0\\
        0& S^T
    \end{pmatrix}$,
and the Stiefel curve $\gamma$ can be written in the familiar form
    \[
        \gamma(t) = UR\cos(t\Sigma)R^T + U_\bot S \sin(t\Sigma) R^T
    \]
known from \cite[Equation 2.65]{EdelmanAriasSmith:1999}, with the sine and cosine functions applied element--wise on then diagonal of $\Sigma$.

A geodesic $t\to c(t)$ with $\|\dot c(t)\|_0\equiv 1$ is called {\em an arc-length geodesic} or simply a {\em unit--speed geodesic}.
For unit--speed geodesics, distances covered by $t$ in the parameter domain are exactly as long as the corresponding part of the trajectory on the manifold.
A Grassmann geodesic in the form of \eqref{eq:Grassmann_proj_geo} is unit--speed in the canonical metric \eqref{eq:canonical_metric} if and only if $\normF{\Delta} = 1$. 

For submanifolds $\mcM$ that are embedded in a Euclidean space, the Riemannian distance $\dist(p,q)$, $p,q\in\mcM$ can be related to the Euclidean distance by elementary geometric means.
The following statement can be inferred from the appendix of \cite{Attali2007}.
\begin{lemma}[{\cite[Property I]{Attali2007}}]
\label{lem:Eulcid2EmbeddedMan_distance}
    Consider $\R^d$ with inner product  $\langle \cdot, \cdot\rangle$ and associated metric $\|\cdot\|=\sqrt{\langle \cdot, \cdot\rangle}$.
    Let $\mcM \subseteq \R^d$ be an embedded submanifold with the Riemannian metric induced by 
    $\langle \cdot, \cdot\rangle$.\\
    Let $p,q\in \mcM$ be such that there is a unit--speed geodesic $c:[0,l] \to \mcM$ with $c(0)=p$ and $c(l)=q$ and Riemannian distance $l=\dist(p,q)$.
    Let $\hat \kappa  = \max_{t\in [0,l]}  \{\|\ddot c(t)\|\}$. \\
    If $\dist(p,q))<\frac{\pi}{\hat \kappa}$ then
    \begin{equation*}
       \dist(p,q) \geq \|p-q\| \geq \frac{2}{\hat \kappa} \sin(\frac{\hat \kappa}2 \dist(p,q)).
    \end{equation*}
\end{lemma}
For an embedded submanifold $\mcM\subset \R^d$ under the Euclidean metric, a manifold curve $c:t\to c(t)\in\mcM$ can also be considered as a classical Euclidean space curve in $\R^d$.
If $c$ is unit--speed, then the quantity $\kappa(t) =\|\ddot c(t)\|$ is the Euclidean curvature of the space curve $c$ at time $t$.
In the language of Frenet curves, $\kappa(t)$ is  the first Frenet curvature \cite[Theorem 2.13]{kuhnel:2015}.
In \cite{Attali2007}, the term `sectional curvature' is used to denote $\kappa$, but this is not consistent with the general concept of sectional curvatures on Riemannian manifolds, see, e.g., \cite[Definition 6.4]{kuhnel:2015}.

\section{Maximum volume coordinates for interpolation on the Grassmann manifold}
\label{sec:MaxVolCoords}

The classical approach to interpolation on Riemannian manifolds works with Riemannian normal coordinates \cite{Amsallem2010, Zim21}. In this section, we present interpolation in an alternative set of local coordinates, where we leverage that the Grassmann manifold of orthogonal projectors features an atlas of coordinate charts and associated parameterizations that are given in closed form and that do not rely on matrix factorizations \cite[Appendix C.4]{HelmkeMoore1994}, \cite[Section 6]{Bendokat:2024}. 

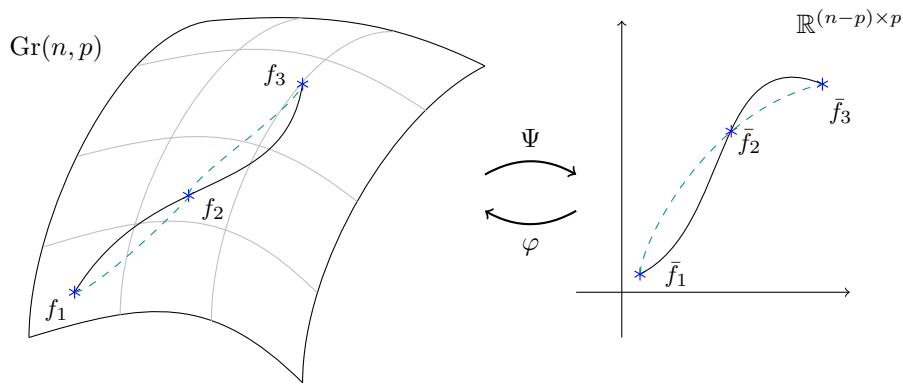
\begin{figure}[!h]
    \centering
    \begin{tikzpicture}[scale=1.2]

    \begin{scope}
    	
    	% the manifold
    	\node[] at (0.3, 3.2) { $\textnormal{Gr}(n,p)$};

        \draw (0,0) ..controls (0,1.3) and (1,3.4).. (2,3.5);
        \draw (0,0) ..controls (1.3,0.4) and (2,0.5).. (3,-0.5);
        \draw (3,-0.5) ..controls (3,0.8) and (4,2.5).. (5,3);
        \draw (2,3.5) ..controls (1.3+1.9,0.4+3.2) and (2+2,0.5+3).. (5,3);

        % A grid
        % horizontal
        \draw[lightgray] (0.155,1) ..controls (1.455,1.4) and (2.155,1.5).. (3.155,0.5);
        \draw[lightgray] (0.53,2) ..controls (1.455+0.385,2.35) and (2.155+0.385,2.35).. (3.155+0.445,1.5);
        \draw[lightgray] (1.18,3) ..controls (1.455+0.65+0.385,3.3) and (2.155+0.65+0.385,3.3).. (3.155+0.65+0.555,2.5);

        % vertical
        \draw[lightgray] (1,0.25) ..controls (1,1.3+0.22) and (2,3.4+0.015).. (3,3.53);
        \draw[lightgray] (2,0.18) ..controls (2,1.3+0.22) and (3,3.4+0.015).. (4,3.38);

        % The curve 
        \draw[] (0.5,0.5) ..controls (1.35,1.9) and (2.8,1.5).. (3,2.8);

        % Points on the curve
        \draw[mark options={color=blue}]
        plot[mark=asterisk, only marks] coordinates {(0.5,0.5) (1.75,1.57) (3,2.8)}; 
        \node[] at (0.3,0.3) {\small $f_1$};
        \node[] at (2.02,1.42) {\small $f_2$};
        \node[] at (2.7,2.9) {\small $f_3$};
        
        % Arrows displaying the maps
        \draw[thick,->] (5,1.8) to [bend left] (6,1.8);
        \node[] at (5.5,2.20) {$\Psi$};
        \draw[<-,thick] (5,1.4) to [bend right] (6,1.4);
        \node[] at (5.5,1) {$\varphi$};

        % Local coordinate system
        \draw[->] (6,0.5) -- (9,0.5);
        \draw[->] (6.5,-0) -- (6.5,3.5);
        \node[] at (9, 3.5) { $\mathbb{R}^{(n-p)\times p}$};

        %The curve in local coordinates
        \draw[] (6.7,0.7) ..controls(7.7,1.2) and (7.5,3.3).. (8.7,2.8);

        % Points on the curve in local coordinates
        \draw[mark options={color=blue}]
        plot[mark=asterisk, only marks] coordinates {(6.7,0.7) (7.7,2.28) (8.7,2.8)}; 
        \node[] at (7.1,0.7) {\small $\bar f_1$};
        \node[] at (7.9,2.16) {\small $\bar f_2$};
        \node[] at (8.9,2.5) {\small $\bar f_3$};

        % Add interpolation curves
        \draw[dashed,ForestGreen] (0.5,0.5) ..controls(0.5+0.25,0.5) and (1.75+0.25,1.57+0.2).. (1.75,1.57);
        \draw[dashed,ForestGreen] (1.75,1.57) ..controls(1.75,1.57+0.2) and (2.9,2.5).. (3,2.8);

         % Add interpolation curve 
         \draw[dashed,ForestGreen] (6.7,0.7) ..controls(6.7,1.) and (7.4,2.5).. (8.7,2.8);

         %\draw[dashed,ForestGreen] (7.7,2.28)..controls(7.7,2.28) and (8.7,2.8).. (8.7,2.8);
        
    \end{scope}

\end{tikzpicture}
    \caption{Concept drawing of interpolation in local coordinates with chart $\Psi$ and parameterization $\varphi=\Psi^{-1}$.
    `Preprocessing': 
    The data points $f_1,f_2,f_3\in \Gr(n,p)$ (blue stars, left)
    of the true function on $\Gr(n,p)$ (solid black, left)
    are mapped to their coordinate images $\bar f_1,\bar f_2,\bar f_3\in \R^{(n-p)\times p}$ via $\Psi$ (blue stars, right).
    `Interpolation':
    The coordinate interpolant is computed in the coordinate domain in $\R^{(n-p)\times p}$ (dashed green, right).
    `Postprocessing': The interpolant is mapped back to the manifold (dashed green, left) via $\varphi$.}
    \label{fig:interpol_local_coorts}
\end{figure}
Given interpolation data from a differentiable map $f:I\to \Gr(n,p)$, interpolation in local coordinates consists of three steps outlined in \Cref{alg:interpol_loc}. Note that the algorithm is generic, $\Gr(n,p)$ could be replaced with any Riemannian manifold. 

In the next sections, we will describe the tools for interpolation in the matrix decomposition--free local coordinates for $\Gr(n,p)$ mentioned in the introduction. Firstly, we describe the local coordinate charts and local parameterizations and discuss the properties and aspects of computing these maps. We prove that this collection of parametrizations of $\Gr(n,p)$ is a retraction \cite[Chapter 4]{Absilbook:2008}; this type of map is frequently encountered in Riemannian optimization. Next, we discuss two interpolation schemes under these coordinates, namely Lagrange and Hermite interpolation. For comparison purposes, their normal coordinate counterparts are stated in \Cref{sec:Int_in_normal_coords}. 

Secondly, we present an error analysis, in which we estimate the error bounds for each of the numerical operations associated with the three steps in \Cref{alg:interpol_loc}. The {\em preprocessing error} stems from the conditioning of the coordinate chart $\Psi$. It turns out that the chart $\Psi$ under investigation may be arbitrarily ill-conditioned. This issue paves the way for introducing maximum--volume coordinates in \Cref{sec:ill_to_well}. The interpolation step comes with an {\em interpolation error} governed by the user-chosen Euclidean interpolation scheme, and the {\em postprocessing error} depends on the conditioning of the parameterization $\varphi$ and the curvature of $\Gr(n,p)$. 

\begin{algorithm}
    \begin{algorithmic}[1]
        \Require Data set $\mathcal{S}$ sampled from $f:I\to \Gr(n,p)$, $t^*\in I$
        \State \textbf{Preprocessing:} Map the data set $\mathcal{S}$ to a chosen local coordinate domain
        \State \textbf{Interpolate:} Apply a Euclidean interpolation scheme to obtain an interpolant in local coordinates.
        \State \textbf{Postprocessing:} Map the interpolant to the manifold to obtain the manifold interpolant $\hat f(t^*)$ of $f$.
    \end{algorithmic}
    \caption{Manifold interpolation in local coordinates, cf. \Cref{fig:interpol_local_coorts}.}
    \label{alg:interpol_loc}
\end{algorithm}

\subsection{Basic construction and properties}
\label{sec:NonMatCompCoords}

In \cite[Appendix C.4]{HelmkeMoore1994}, \cite[Section 6]{Bendokat:2024}, a covering set of local parameterizations of the Grassmannian is discussed that do not rely on matrix decompositions. Briefly, if $U\in \St(n,p)$ is a basis of a subspace $\mcU$, then $P=UU^T$ is the unique orthogonal projection onto $\mcU$. With $U=\begin{bmatrix}
    U_1\\
    U_2
\end{bmatrix}, U_1\in \R^{p\times p}, U_2\in \R^{(n-p)\times p}$, we obtain 
\begin{equation*}
    P=\begin{pmatrix}
        U_1U_1^T&U_1U_2^T\\
        U_2U_1^T&U_2U_2^T
    \end{pmatrix}:=\begin{pmatrix}
        A     &\hat B^T\\
        \hat B&C
    \end{pmatrix}\in\sym(n).
\end{equation*}
If $A\in \sym(p)$ is invertible, we can send $P$ bijectively to local coordinates in $\R^{(n-p)\times p}$.
More precisely, the following map is a
coordinate chart 
\begin{equation}
\label{eq:loc_coord_psi}
    \Psi:\Gr(n,p)\supset \Omega \to  \R^{(n-p)\times p}, \, \Psi\qty(\begin{pmatrix}
    A&\hat B^T\\
    \hat B&C
\end{pmatrix})= \hat BA^{-1},
\end{equation}
where $\Omega\subset \Gr(n,p)$ is a (relative) open set. The associated parameterization $\varphi:=\Psi^{-1}$ sends a given coordinate matrix $B\in \R^{(n-p)\times p}$ to a projector $P=\varphi(B)\in \Gr(n,p)$,
\begin{equation}
\label{eq:loc_coord_varphi}
    \varphi:\R^{(n-p)\times p}\to \Gr(n,p):\varphi(B)=\begin{pmatrix}
        I_p\\
        B
    \end{pmatrix}\left(I_p+B^T B\right)^{-1}\begin{pmatrix}
        I_p& B^T
    \end{pmatrix}.
\end{equation}
The matrix factor in the middle $I_p + B^T B$ is guaranteed to be invertible, because $B^TB$ is positive semi-definite for any $B\in \R^{(n-p)\times p}$.
If $(B^TB)=V\Lambda V^T$ is the EVD, then
$I_p+B^TB= V(I+\Lambda)V^T$. Hence, the eigenvalues are $1+\lambda_i$, $\lambda_i\geq 0$. Therefore, $I_p + B^TB$ is strictly positive definite and $\varphi$ is globally defined. The condition number is
$\cond_2(I_p+B^TB) = \|I_p+B^TB\|_2\|(I_p+B^TB)^{-1}\|_2= \frac{1+\lambda_{\max}}{1+\lambda_{\min}}\leq  1+\lambda_{\max}$.

For an embedded submanifold described by local parameterizations, the tangent space is the image of the differential of the parameterization.
More precisely, if $\varphi$ is a parameterization with $\varphi(B) = P$, then $\T_P\Gr(n,p) = \colspan(D\varphi_B)$, cf.
\cite[Def. 3.1.1]{Klingenberg:1978}. 
For the special point $B=0\in\R^{(n-p) \times p}$ that corresponds to the canonical projector onto the first $p$ coordinate axes,
$
    \varphi(0)= P_0
    = 
    \begin{pmatrix}
        I_p & 0\\
        0   & 0
    \end{pmatrix},
$
we obtain

\begin{align}
\nonumber
    d\varphi_{0}(\Delta)&=\begin{pmatrix}
        0\\
        \Delta
    \end{pmatrix}\begin{pmatrix}
        I_p&0
    \end{pmatrix}+\begin{pmatrix}
        I_p\\
        0
    \end{pmatrix}\eval{\deriv}_{t=0}(I_p+t^2\Delta^T\Delta)^{-1}\begin{pmatrix}
        I_p&0
    \end{pmatrix}+\begin{pmatrix}
        I_p\\
        0
    \end{pmatrix}\begin{pmatrix}
        0&\Delta^T
    \end{pmatrix}\\
    \label{eq:Gr_tanspace_P0}
    &=\begin{pmatrix}
        0&\Delta^T\\
        \Delta&0
    \end{pmatrix}\quad \implies
    \T_{P_0}\Gr(n,p) = \left\{
    \begin{pmatrix}
        0&\Delta^T\\
        \Delta&0
    \end{pmatrix}~\mid~ \Delta\in \R^{(n-p)\times p}
    \right\}.
\end{align}
With the group action $\Phi$ from \eqref{eq:projector_group_action}, we can translate the local parameterization $\varphi$ that parameterizes a manifold domain around $P_0$ to obtain local parameterizations around any other point on $\Gr(n,p)$:
Let $P\in \Gr(n,p)$. Then there is $Q=Q(P)\in O(n)$ such that $P=QP_0Q^T$.
The mapping
$$
    \varphi^Q:\R^{(n-p)\times p}\to \text{im}(\varphi^Q)\subset\Gr(n,p),~ B\mapsto Q\varphi(B)Q^T
$$
is a parameterization with $P\in \text{im}(\varphi^Q)$ and coordinate chart
\[		\Psi^Q\big\vert_{\text{im}(\varphi^Q)}:\text{im}(\varphi^Q)\to \R^{(n-p)\times p}, \tilde P\mapsto \Psi(Q^T \tilde PQ).
\]

It is clear that the differential $d\varphi^Q (\Delta)=Qd\varphi_0(\Delta)Q^T$ is a bijection when viewed as a map $d\varphi^Q:\R^{(n-p) \times p}\to \T_P\Gr(n,p)$. 

We recall from \cite[Chapter 4]{Absilbook:2008}, \cite[Section 3.6]{boumal2023} that a retraction $R$ on a Riemannian manifold $\mcM$ is a smooth mapping from the tangent bundle $\T\mcM$ to $\mcM$ such that
\begin{itemize}
\item for all $p\in\mcM$, $R_p: \T_p\mcM \to \mcM$
\item $R_p(0)=p$, where $0$ is the zero of $\T_p\mcM$.
\item The differential of $R_p$ at $0$ is the identity, $d(R_p)_0 = \textnormal{id}_{\T_p\mcM}$.
\end{itemize}

In both the books \cite{Absilbook:2008}, \cite{boumal2023},
it is only for notational convenience that the full tangent bundle $\T\mcM$ is stated as the domain of definition for the retraction. 
The concept works the same, if all $R_p$ are defined on open subsets around $0$ in the associated tangent spaces $\T_p\mcM$, $p\in \mcM$. In fact, without allowing for this relaxation, not even the Riemannian exponential would count as retraction, while the original motivation to consider retractions is to have `cheap' approximations of the Riemannian exponential.
\begin{theorem}
\label{lem:loc_paras_retraction}
The collection of parameterizations $\{\varphi^Q\mid Q\in O(n)\}$ above gives rise to a retraction on $\Gr(n,p)$.
\end{theorem}
\begin{proof}
Note that 
\begin{equation*}
    \{\Delta\mid \Delta\in\R^{(n-p)\times p}\}\cong \left\{\begin{pmatrix}
        0&\Delta^T\\
        \Delta&0
    \end{pmatrix}\mid \Delta\in\R^{(n-p)\times p}\right\}=\T_{P_0}\Gr(n,p),
\end{equation*}
and we may consider the `base map' $\varphi$ immediately as a map
\[
    \hat \varphi: \T_{P_0}\Gr(n,p) \to \Gr(n,p),
    ~ \begin{pmatrix}
        0&\Delta^T\\
        \Delta&0
    \end{pmatrix} \mapsto \varphi(\Delta).
\]
The same calculation as in \eqref{eq:Gr_tanspace_P0} shows that 
$d(\hat \varphi)_0(X) = X$ for all $X\in \T_{P_0}\Gr(n,p)$.

For an arbitrary parameterization $\varphi^Q$ around $P = QP_0Q^T\in\Gr(n,p)$, the transition to make $\varphi^Q$ a tangent space map is as follows.
Let $X\in \T_P\Gr(n,p)$. There exists $\Delta\in\R^{(n-p)\times p}$ such that $X = Q\begin{pmatrix}
        0&\Delta^T\\
        \Delta&0
    \end{pmatrix}Q^T$.
Therefore
$
   \begin{pmatrix}
    0\\
    \Delta
\end{pmatrix} = Q^TXQ
\begin{pmatrix}
    I_p\\
    0
\end{pmatrix}
$, 
The parameterization $\varphi^Q$ acts essentially on the block matrix  $$\begin{pmatrix}
    I_p\\
    \Delta
\end{pmatrix}
= Q^TXQ
\begin{pmatrix}
    I_p\\
    0
\end{pmatrix} + \begin{pmatrix}
    I_p\\
    0
\end{pmatrix}
= (Q^T(X + I_n)Q)\begin{pmatrix}
    I_p\\
    0
\end{pmatrix}.$$
Hence,
\begin{align}
\nonumber
    &\hat \varphi^Q: \T_P\Gr(n,p)\to \Gr(n,p), \quad
    \hat \varphi^Q(X) = Q\varphi\left(\Delta\right)Q^T\\
    \nonumber
    =& 
    Q 
    \Biggl(
    (Q^T(X+I)Q)\begin{pmatrix}
    I_p\\
    0
    \end{pmatrix} 
    \left[
        \begin{pmatrix}
            I_p &    0
        \end{pmatrix}
        (Q^T(X+I)Q)
    (Q^T(X+I)Q)\begin{pmatrix}
            I_p\\
            0
        \end{pmatrix}
    \right]^{-1}\\
    \nonumber
    & 
    \begin{pmatrix}
            I_p &    0
    \end{pmatrix}
    (Q^T(X+I)Q)
    \Biggr)
    Q^T\\
\label{eq:loc_para_tanspace}
=&
(X+I)Q\begin{pmatrix}
    I_p\\
    0
\end{pmatrix} 
\left[
\begin{pmatrix}
    I_p &    0
\end{pmatrix}
(Q^T(X+I)^2Q)\begin{pmatrix}
    I_p\\
    0
\end{pmatrix}
\right]^{-1}
\begin{pmatrix}
    I_p &    0
\end{pmatrix}
Q^T(X+I).
\end{align}

Now, we can compute $d(\varphi^Q)_0(X)=\eval{\deriv}_{t=0}\varphi^Q(tX)$ by differentiating \eqref{eq:loc_para_tanspace}.
This yields
\begin{align}
\nonumber
  d(\varphi^Q)_0(X)  
    =& 
  XQ\begin{pmatrix}
    I_p\\
    0
\end{pmatrix} 
\begin{pmatrix}
    I_p &    0
\end{pmatrix}
Q^T + 
  Q\begin{pmatrix}
    I_p\\
    0
\end{pmatrix} 
\begin{pmatrix}
    I_p &    0
\end{pmatrix}
Q^TX\\
\nonumber
& - 2Q\begin{pmatrix}
    I_p\\
    0
\end{pmatrix} 
\begin{pmatrix}
    I_p &    0
\end{pmatrix}
Q^T
X
Q\begin{pmatrix}
    I_p\\
    0
\end{pmatrix} 
\begin{pmatrix}
    I_p &    0
\end{pmatrix}
Q^T \\
\nonumber
& = XQP_0Q^T + QP_0Q^T X- 2QP_0Q^TXQP_0Q^T\\
\nonumber
& = XP + PX - 2PXP = XP + PX = X,
\end{align}
because of \eqref{eq:GrTan_characterization}.

It remains to show that $\hat \varphi^{Q}: \T_P\Gr(n,p)\to \Gr(n,p)$ depends smoothly on the base point $P$.
The coordinate chart $\hat \varphi^Q$ depends on $P$ through $Q(P)$.
Because $\Gr(n,p)$ is geodesically complete, any two points $P,\tilde P \in \Gr(n,p)$ can be connected by a geodesic.\footnote{For the argument, local paths would already be sufficient.}
Let $y:t\to y(t)\in \Gr(n,p)$ be a geodesic with $y(0) = P$, $y(1) = \tilde P$. The horizontal lift is a horizontal Stiefel geodesic $U:t\to U(t)\in \St(n,p)$ such that 
$y(t) = U(t)U(t)^T$ and $P = U(0)U(0)^T$, $\tilde P= U(1)U(1)^T$.
By the Gram--Schmidt process, there is an extension $Q(P(t)) = \begin{pmatrix}
    U(t), U_\bot(t)
\end{pmatrix}\in \textnormal{SO}(n)$, where $U_\bot(t)\in\St(n, (n-p))$ depends smoothly on $\qty(t, U(t))$. It holds $Q(P(0))\begin{pmatrix}
        I_p & 0\\
        0 &0
    \end{pmatrix}Q(P(0))^T = U(0)U(0)^T = P$
and
$Q(P(1))\begin{pmatrix}
        I_p & 0\\
        0 &0
    \end{pmatrix}Q(P(1))^T = U(1)U(1)^T = \tilde P$.
Hence, $Q$ can be constructed to depend smoothly on $P$.
As a consequence, the corresponding chart $\hat \varphi^{Q} = \hat \varphi^{Q(P)}$ also depends smoothly on $P$. 
\end{proof}

The Riemannian subspace distance (see \Cref{app:subspace_dist}) between $P_0$ and any $\varphi(B)$ is encoded in the SVD of the coordinate matrix $B$.
\begin{lemma}
Let $B\in\R^{(n-p)\times p}$ be arbitrary and let $Q\Sigma R^T=B$ be the compact SVD of $B$ with $\Sigma = \diag(\sigma_1,\ldots,\sigma_p)$. Then
\[
    \dist( P_0, \varphi(B)) = \sqrt{\sum_{k=1}^p \arctan(\sigma_k)^2}.
\]
The distance is bounded above by $\dist( P_0, \varphi(B))\leq \sqrt{p}\arctan(\|B\|_2)$.
\end{lemma}
\begin{proof}
Stiefel representatives of $P_0$ and $\varphi(B)= \begin{pmatrix}
        I_p\\
         B
    \end{pmatrix}\left(I_p+B^TB\right)^{-1}\begin{pmatrix}
        I_p&B^T
    \end{pmatrix}$
are 
\[
   U_0 := \begin{pmatrix}
        I_p\\
         0
    \end{pmatrix}, \quad
    \tilde U :=\begin{pmatrix}
        I_p\\
         B
    \end{pmatrix} T^{-1},
\]
where $T$ is a matrix square root of $I_p+B^TB=T^2$ (see also point (A) in the following subsection).
When using the SVD $B=Q\Sigma R^T\in \R^{(n-p)\times p}$, a matrix square root is obtained by 
$T = R \sqrt{I_p + \Sigma^2} R^T$.
To obtain the principal angles (\Cref{app:subspace_dist}), compute
\[
    U_0^T \tilde U = T^{-1} = R (I_p + \Sigma^2)^{-\frac12} R^T.
\]
By \eqref{eq:subspace_dist}, the subspace distance is
\begin{align*}
    \dist(P_0, \varphi(B))^2 &= \sum_{k=1}^p \qty(\arccos\left(\frac{1}{\sqrt{1+\sigma_k^2}}\right))^2 
     = \sum_{k=1}^p (\arccos\left(\cos(\arctan(\sigma_k))\right))^2\\
     &= \sum_{k=1}^p (\arctan(\sigma_k))^2.
\end{align*}
Hence, an upper bound for the subspace distance is 
\[ 
\dist(P_0, \varphi(B)) = \sqrt{\sum_{k=1}^p \qty(\arctan(\sigma_k))^2}\leq \sqrt{p}\arctan\qty(\|B\|_2).
\]
\end{proof}

\subsection{Numerical processing of Grassmann data}
\label{sec:numprocessing}
Formally, the local coordinates for the Grassmannian $\Gr(n,p)$ reviewed in \Cref{sec:NonMatCompCoords} work with projectors, i.e. matrices in $\sym(n)\subset \R^{n\times n}$.
Given a subspace $\mathcal{U}\in\Gr(n,p)$ with orthonormal basis $U\in \St(n,p)$, the orthogonal projector is $UU^T\in \sym(n)$.
Yet, the most important practical applications that feature Grassmann data are those that emphasize and exploit the low-rank structure of $UU^T$. Especially, if $n\gg p$, an $(n\times n)$ matrix must never be formed. This means that in an algorithmic implementation, the coordinate chart $\Psi$ receives a matrix representative $U= \begin{pmatrix}
    U_1\\
    U_2
\end{pmatrix}\in \St(n,p)$ as an input, where the subblocks are of dimension $U_1\in \R^{p\times p}$, $U_2\in \R^{(n-p)\times p}$.
The output is evaluated as
\begin{equation}
\label{eq:MD-free-practice}
    \Psi(UU^T) = \Psi
    \begin{pmatrix}
    U_1U_1^T & U_1 U_2^T\\
    U_2U_1^T & U_2U_2^T
    \end{pmatrix}
    = U_2U_1^T (U_1U_1^T)^{-1} = U_2U_1^{-1}\in\R^{(n-p)\times p}.
\end{equation}
Thus, the largest dimension of matrices formed in code is $(n\times p)$.
The upper $(p\times p)$--subblock $U_1$ may be replaced with any subselection of $p$ rows $U$ that are linearly independent, with the subblock $U_2$ then being chosen as the complementary subset of rows of $U$.

The associated parameterization $\varphi$ can be evaluated as follows:
\begin{itemize}
\item[(A)] Using matrix square roots: 
Given $B\in \R^{(n-p)\times p}$,
compute a matrix square root $T$ such that
\[
    (I_p + B^TB) = T^2.
\]
The associated $(n\times p)$ Stiefel representation of the projector $\varphi(B)$
is $U:= \begin{pmatrix}
    I_p\\
    B
\end{pmatrix} T^{-1}$. 
Obviously, $U^TU=I_p$.
\item[(B)] Using the Cholesky decomposition:
  Compute $L\in \R^{p\times p}$, lower triangular, such that
  \[
    (I_p + B^TB) = LL^T
  \]
  The associated $(n\times p)$ Stiefel representation of the projector $\varphi(B)$
is $U:= \begin{pmatrix}
    I_p\\
    B
\end{pmatrix} L^{-T}$. 
Again, $U^TU=I_p$.
\end{itemize}

When processing derivative data in Stiefel representation, it is important to note that if $U(t)$ is a Stiefel representation of a curve $U(t)U(t)^T=P(t)\in \Gr(n,p)$, then $\dot U(t)$ is not necessarily a {\em horizontal} Stiefel representation of $\dot P(t)$. To obtain a horizontal Stiefel representation of $\dot P(t)$, compute $\Delta(t)=\qty(\dot U(t)U(t)^T-U(t)\dot U(t)^T)U(t)$, which is the horizontal lift of $\dot P(t)$. 

\subsection{Lagrange- and Hermite interpolation}

We are now ready to discuss interpolation. We assume throughout the section that for a given set of Lagrange-- or Hermite data $\mcS=\qty{t_i,f_i,f_i'}_{i=1}^d$, all the points $f_i\in \Gr(n,p)$ are in the domain of a single chart $\Psi:\Gr(n,p)\to \R^{(n-p)\times p}$. To map the Grassmann data to the local coordinate domain, we use that the local chart $\Psi$ is a local diffeomorphism, with differential $d\Psi_{P}:\T_P\Gr(n,p)\to  \T_{\Psi(P)} \R^{(n-p)\times p}\simeq \R^{(n-p)\times p}$
being a vector space isomorphism. Hence points and tangent vectors can be mapped bijectively to the local coordinate domain, and we obtain $\bar \mcS=\qty{t_i,\bar f_i,\bar f_i'}_{i=1}^d\subset\R^{(n-p)\times p}$. Similar to interpolation in normal coordinates, the vector space structure of $\R^{(n-p)\times p}$ allows us to apply any interpolation scheme from classical interpolation theory. The well--known Lagrange interpolation polynomial is stated in \Cref{sec:Int_in_normal_coords}. For Hermite interpolation, we use the formulation as in \cite[Section 8.5]{Alfio:2007}. Considering the special case of two data samples $\qty{t_i,\bar f_i,\bar f_i'}_{i=1}^2$, we construct the cubic Hermite polynomial
\begin{equation*}
    \gamma(t)=L_{00}(t)\bar f_0+L_{10}(t)\bar f_1+L_{01}(t)\bar f_0'+L_{11}(t)\bar f_1'.
\end{equation*}
The coefficient functions $L_{00},\dots,L_{11}$ are listed in \Cref{sec:Int_in_normal_coords}. The resulting Hermite interpolant of $f$ is $\hat f(t^*)=\varphi(\gamma(t^*))$. Replacing $\gamma(t)$ with the Lagrange polynomial based on the data in local coordinates gives similarly the Lagrange interpolant of $f$. 

To map derivative data from $\T_{f_i}\Gr(n,p)$ to the local coordinates, we compute the directional derivative of $\Psi$. 
% -------new text starts here------------
With the quotient map $\Pi: \St(n,p) \to \Gr(n,p), U\mapsto UU^T$ we have the Stiefel counterpart of $\Psi$,
\[
\tilde \Psi : \St(n,p) \to \R^{(n-p)\times p}, \quad U \mapsto \tilde \Psi(U) = (\Psi \circ \Pi)(U).
\]
It holds $d\tilde \Psi_U:\T_U\St(n,p) \to \R^{(n-p)\times p}$ and 
\[
    d\Psi_{\Pi(U)}\qty(d\Pi_U(T))) = d\tilde \Psi_U(T).
\]
Let $U(t)=\begin{pmatrix}
    U_1(t)\\
    U_2(t)
\end{pmatrix}$ be a Stiefel curve with $U(0) = U$, $\dot U(0) = T=\begin{pmatrix}
    T_1\\
    T_2
\end{pmatrix}$ and consider the corresponding Grassmann curve $P(t)=\Pi(U(t)) = U(t)U(t)^T$ with  $P(0)=P,\dot P(0)=X =D\Pi_U(T)= TU^T + UT^T$.
Then
\begin{equation}\label{eq:d_psi}
    d\Psi_P(X) = d\tilde \Psi_U(T) = 
    \eval{\frac{\dd}{\dd t}}_{t=0} \tilde \Psi(U(t)) = T_2 U_1^{-1} - U_2 U_1^{-1} T_1 U_1^{-1}\in \R^{(n-p)\times p}.
\end{equation}

This formula allows one to compute the Grassmann derivative $d\Psi_P(X)$ by using Stiefel representatives with matrix blocks of size $(n-p)\times p$ and $p\times p$.
Since $d\Psi_P$ has the inverse map $d\varphi_{\Psi(P)}$, we have that Hermite interpolating $f:I\to \Gr(n,p)$ on the sample data $\mcS=\qty{t_i,f_i,f_i'}_{i=1}^d$ is equivalent to Hermite interpolating $\bar f=\Psi\circ f$ on the sample data $\bar\mcS=\qty{t_i, \bar f_i, \bar f_i'}_{i=1}^d= \qty{t_i, \Psi(f_i), d\Psi_{f_i}(f_i')}_{i=1}^d$, and it is by this construction that the Hermite polynomial $\gamma (t)$ stated above is the Hermite interpolant in local coordinates of $\bar f$. 

Comparing to interpolation in Riemannian normal coordinates, mapping derivative data to the local coordinate domain does not require a finite differences--approxi\-mation. 

\subsection{Data processing error and maximum-volume coordinates}\label{sec:ill_to_well} 

To motivate the following investigations, consider the Stiefel matrix $U=\begin{pmatrix}
    0_p&
    I_p&
    0_{n-2p\times p}
\end{pmatrix}^T$. Computing the local coordinates associated with the projector $P=UU^T$ without a row--selection strategy is impossible. In practice this type of matrix does hardly arise, but Stiefel representations $U=\begin{pmatrix}
    U_1\\U_2
\end{pmatrix}$ with $U_1$ being ill-conditioned are not unrealistic. This causes points which are close on the Grassmann manifold to be mapped to local coordinates that are very far apart (see \Cref{fig:manifold_dist_vs_local_dist}). It may even make it impossible to obtain such coordinates due to numerical errors,
and interpolation of the coordinate images becomes futile.
\begin{figure}[!ht]
    \centering
    \includegraphics[width=0.8\linewidth]{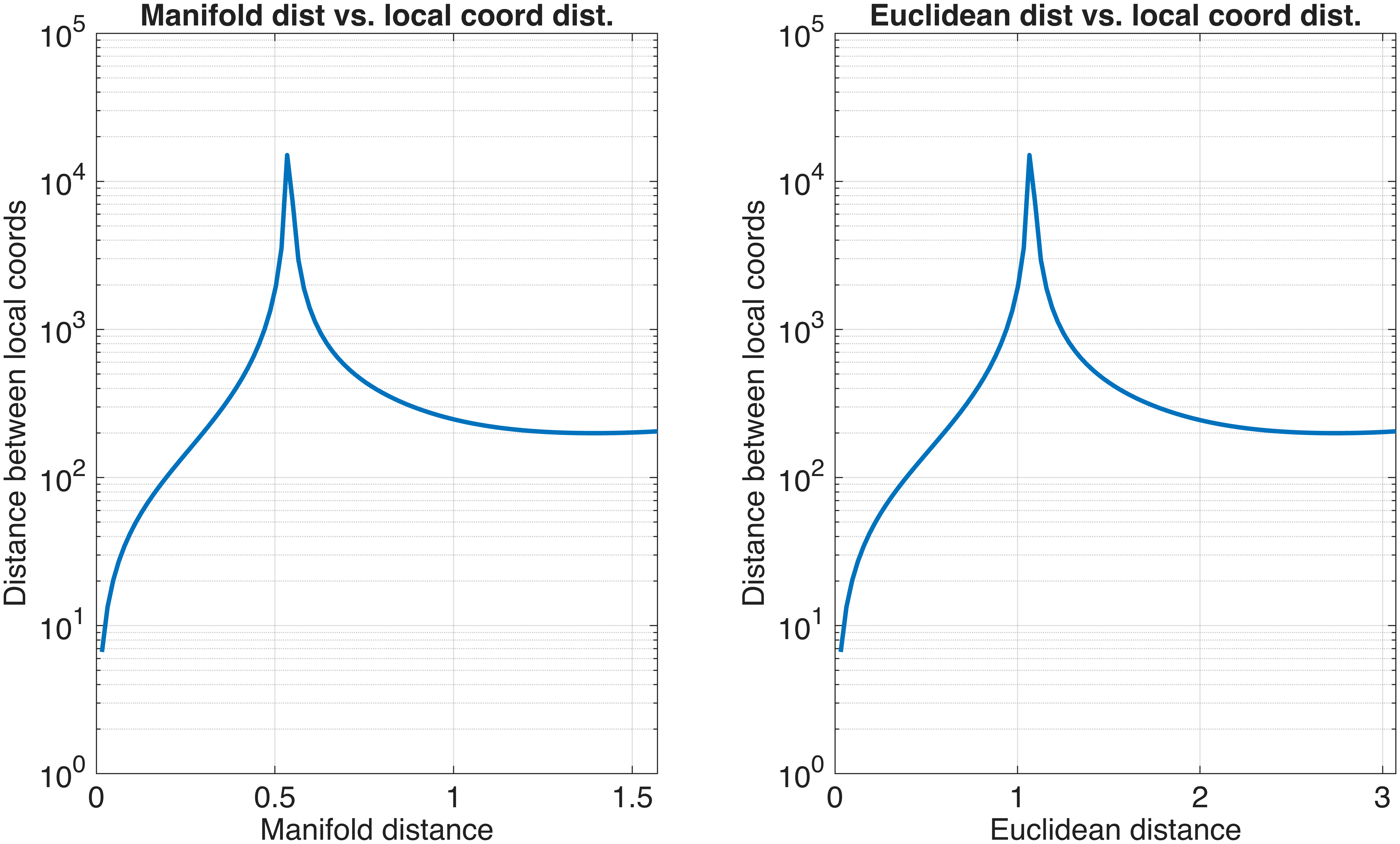} 
    \caption{Distance between a random point $U\in \Gr(n,p)$ (in Stiefel representation with $n=1000,p=20$, obtained from a QR--decomposition of a random $n\times p$ matrix) and synthetic points $\tilde U_t = \Exp_U(t\Delta)$, where $\Delta$ is a fixed tangent vector of unit norm $\|\Delta\|_0=1$ and $t\in[0,\pi/2]$ is varied.
By construction, the Riemannian distance is $\dist(U,\tilde U_t)=t$ (x-axis, left plot); the Euclidean distance is $\normF{U-\tilde U_t}$ (x-axis, right plot). 
    The distance in local coordinates is $\normF{\Psi(U)-\Psi(\tilde U_t)}$ (y-axis, both plots).
    }
    \label{fig:manifold_dist_vs_local_dist}
\end{figure}

The next result bounds the distance between the outputs of the coordinate chart. 
\begin{lemma}\label{lem:psi_dist}
    Given two Stiefel representations 
    $U= \begin{pmatrix}
        U_1\\
        U_2
    \end{pmatrix},
    V= \begin{pmatrix}
        V_1\\
        V_2
    \end{pmatrix}$ of projectors $P_U,P_V\in \Gr(n,p)$,
    the distance of their coordinate images under $\Psi$ from \eqref{eq:loc_coord_psi} is bounded by
    \begin{equation*}
        \normF{\Psi(P_U)-\Psi(P_V)}\leq \sqrt{\normF{U_1^{-1}}^2 - p}
        + 
        \sqrt{\normF{V_1^{-1}}^2 - p}.
    \end{equation*}
\end{lemma}
\begin{proof}
From \eqref{eq:MD-free-practice},
\[
    \normF{\Psi(P_U)-\Psi(P_V)} = \normF{U_2U_1^{-1}- V_2V_1^{-1}} \leq
    \normF{U_2U_1^{-1}} + \normF{V_2V_1^{-1}}.
\]
%\textbf{Estimate 1:}
Because $U\in \St(n,p)$, $I = U_1^TU_1 + U_2^TU_2$. Hence, 
\[
  \norm{U_2U_1^{-1}}_F^2 
  = \trace U_1^{-T}U_2^TU_2U_1^{-1} 
  = \trace U_1^{-T}(I_p - U_1^TU_1)U_1^{-1}  
  = \norm{U_1^{-1}}_F^2 - p.
\]

Likewise for $V$.
\end{proof}
\begin{remark}
    \label{rem:max_vol_bound}
    The Lemma above shows that if $\normF{U_1^{-1}}$ and $\normF{V_1^{-1}}$ are small, then the upper bound on the spread is small as well. 
    To optimize the bound in the lemma is to minimize the norm of the inverse of the block $U_1$. Exactly the same challenge face CUR-matrix decompositions (also called pseudo-skeleton matrix approximations) \cite{Goreinov:1997}
    and the discrete empericial interpolation method (DEIM), \cite{Chaturantabut:2010,DrmacGugercin:2015}, and it is closely related to rank-revealing matrix facorizations
    \cite{ChandrasekaranIpsen:1994, GuEisenstat:1996, damle:2024}.
    In fact, a minor modification of \cite[Lemma 2.1]{Goreinov:1997} shows that if the block $U_1$ is formed by $p$ rows that maximize the `volume' $|\det(U_1)|$, then 
    \[
    \normF{U_1^{-1}}\leq 
    \sqrt{ p(n-p) + p}.
    \]
    For a coordinate chart $\Psi$ that is based on the optimal choice of the subblock $U_1$, we thus obtain
    \[
        \normF{\Psi(P_U)} \leq \sqrt{ p(n-p) + p - p} = \sqrt{p(n-p)}.
    \]
    %Related to a double-precision machine epsilon, this is a rather moderate value.
    %For example, even for huge matrices with  $n=1$mio. and $p=1000$, the bound is of the order $10^4$.
\end{remark}

These observations suggest to optimize the subblock selection.
This is equivalent to computing a permutation matrix $\mcP\in \R^{n\times n}$ such that 
\begin{equation}\label{eq:perm_stiefel}
    \mcP\begin{pmatrix}
        U_1\\
        U_2
    \end{pmatrix}=\begin{pmatrix}
        \tilde U_1\\
        \tilde U_2
    \end{pmatrix}, \quad \normF{\tilde U_1^{-1}} <\normF{U_1^{-1}}.
\end{equation} 
Computing the local coordinates of the permuted blocks is the same as composing the local coordinate chart with a map $\Upsilon_{\mcP}:\Gr(n,p)\to \Gr(n,p), P\mapsto \mcP P\mcP^T$ to obtain the modified chart $\Psi\circ \Upsilon:\Gr(n,p)\to \R^{(n-p)\times p}$. 

By \cite[Lemma 2.1]{Goreinov:1997}, a practical approach to optimize the chart is to use a volume maximizer for selecting the subblock $U_1$.
We will call the resulting chart a {\em maximum--volume local coordinate chart} or {\em MV-chart} for short.

The literature knows a large variety of methods for this task, see \cite{damle:2024} and we have no intention (and no need) to decide which method works best.
In the numerical experiments,  we apply the maxvol method of \cite[Algorithm 1]{Goreinov:2010}
to compute the permutation matrix $\mcP$.
\begin{algorithm}
    \begin{algorithmic}[1]
        \Require Hermite data $\qty{t_i,f_i,f_i'}_{i=1}^d\in I\times \Gr(n,p) \times T\Gr(n,p)$
        \For{$i=1$ to $d$}
        \State Apply the \texttt{maxvol}--algorithm of \cite{Goreinov:2010} on $f_i$ to obtain $\mcP_i$
        \EndFor
        \State Select $\mcP$ among $\qty{\mcP_1,\dots,\mcP_d}$ 
        that minimizes $\max_{i=1,\ldots,d} \{\| \left( f_i[1:p,1:p]\right)^{-1}\|_F\}$.
    \end{algorithmic}
    \caption{Compute permutation $\mcP$ and prepare for preprocessing}
    \label{alg:prepare_data}
\end{algorithm}

Given a set of Grassmann data $\qty{f_i}_{i=1}^d$ in Stiefel representation, for consistency in the interpolation procedure, one and the same permutation matrix $\mcP$
must be chosen when preparing the data for preprocessing.
Therefore, the actual task is to find a single permutation matrix such that the upper $p\times p$ blocks of each of the Grassmann samples $f_i, i=1,\dots, d$ are well--conditioned. 

For example in model reduction applications, where the data $\qty{f_i}_{i=1}^d$ often stems from reduced bases for parametric dynamical systems, it is to be expected that the rows that contain essential independent information are located in similar positions throughout the data set.

\subsection{The conditioning of the local parameterizations}
\label{sec:para_cond}

The absolute condition number of a continuously differentiable map $f:\R^n\to \R^m$ is \cite[Theorem 3.1]{Higham:2008:FM}
\begin{equation*}
    \cond(f,X)=\max_{\norm{D}=1}\norm{df_X(D)}.
\end{equation*}
The next lemma provides a bound for the condition number of the local parameterization $\varphi$ from \eqref{eq:loc_coord_varphi}. This result will be used to quantify how the interpolation errors in local coordinates are amplified when forwarded to the manifold.  Mind that $d\varphi_B:\R^{(n-p)\times p} \to \T_{\varphi(B)}\Gr(n,p)\subset\sym(n)\subset \R^{n\times n}$. Since we consider the Grassmannian of rank-$p$ projectors as an embedded submanifold of $(\sym(n), \langle \cdot,\cdot\rangle_0)$ with metric inherited from this space, we are required to compute  $\|d\varphi_B(\Delta)\|_0$ in the canonical metric, i.e., the Euclidean metric up to a factor of $\frac12$, if we want the results to be compatible with the Riemannian measures.
The appropriate norm on the input domain $\R^{(n-p)\times p}$ is the Frobenius norm $\normF{\cdot}$. For the linear maps $d\varphi_B$, we use the operator norm
$\|d\varphi_B\|_{\textnormal{F},0} = \sup_{V\neq 0}\frac{\|d\varphi_B(V)\|_0}{\normF{V}}= \max_{\normF{V}=1}\|d\varphi_B(V)\|_0$.
\begin{lemma}
    \label{lem:cond_varphi}
    Let $B\in\R^{(n-p)\times p}$ with a compact SVD of $B=U\Sigma V^T$ with $\Sigma = \diag(\sigma_1,\ldots,\sigma_p)$.
    The condition number of $\varphi$ at $B$ is bounded by
    \begin{equation}\label{eq:cond_varphi}
        \cond_0(\varphi,B) %=\sup_{\norm{\Delta}_F=1}\norm{D\varphi_B(\Delta)}
        \leq 
        \sqrt{2} \sqrt{\left(\frac{1}{(1+\sigma_{p}^2)^2} + 
    \max_{i=1,\ldots p}\qty{ \frac{\sigma_i^2}{(1+\sigma_{i}^2)^2}}\right)} + 1.
    \end{equation}
    The global worst-case bound is
    \begin{equation}\label{eq:cond_varphi_global}
        M:= \sup_{B\in\R^{(n-p)\times p }}\{\cond_0(\varphi,B)\}
    \leq \sqrt{\frac{5}{2}}+1 \approx 2.5811.
    \end{equation}
\end{lemma}

\begin{proof}
For $B\in \R^{(n-p)\times p}$ write $P=\varphi(B)$ and $S=S(B) = (I_p+B^TB)^{-1}$. The differential in direction $\Delta\in\R^{(n-p)\times p}$, is
    \begin{align}
    \nonumber
    d\varphi_B(\Delta)=&\begin{pmatrix}
        0\\
        \Delta
    \end{pmatrix}S\begin{pmatrix}
        I_p&B^T
    \end{pmatrix}
    %
    %
    %\nonumber
    %&
    -\begin{pmatrix}
        I_p\\
        B
    \end{pmatrix}S(\Delta^TB+B^T\Delta)S\begin{pmatrix}
        I_p&B^T
    \end{pmatrix}
    %
    %
    %\nonumber
    %&
    +\begin{pmatrix}
        I_p\\
        B
    \end{pmatrix}S\begin{pmatrix}
        0&\Delta^T
    \end{pmatrix}\\
    \nonumber
    =&\begin{pmatrix}
        0&0\\
        \Delta S&\Delta SB^T
    \end{pmatrix}
    %
    %
    %\nonumber
    %&
    -\begin{pmatrix}
        I_p\\
        B
    \end{pmatrix}S\begin{pmatrix}
        I_p&B^T
    \end{pmatrix}\begin{pmatrix}
        0&\Delta^T\\
        \Delta&0
    \end{pmatrix}\begin{pmatrix}
        I_p\\
        B
    \end{pmatrix}S\begin{pmatrix}
        I_p&B
    \end{pmatrix}
    +\begin{pmatrix}
        0&S\Delta^T\\
        0&BS\Delta^T
    \end{pmatrix} \nonumber \\
    \label{eq:varphi_diff}
    =&\begin{pmatrix}
        0&0\\
        \Delta S&\Delta S B^T
    \end{pmatrix}-P\begin{pmatrix}
        0&\Delta^T\\
        \Delta&0
    \end{pmatrix}P+\begin{pmatrix}
        0&S \Delta^T\\
        0&BS \Delta^T
    \end{pmatrix}.
\end{align}
To compute a bound for $\cond_0(\varphi,B)=\norm{d\varphi_B}_{\textnormal{F},0}$, let $\Delta\in \R^{(n-p)\times p}$ with $\normF{\Delta}=1$. 
 Because $P=\varphi(B)$ is an orthogonal projector, the middle term in \eqref{eq:varphi_diff} is bounded by 
 \begin{equation*}
     \norm{P\begin{pmatrix}
     0&\Delta^T\\
     \Delta&0
 \end{pmatrix}P}_0^2\leq 
     \norm{\begin{pmatrix}
     0&\Delta^T\\
     \Delta&0
 \end{pmatrix}}^2_0
 = \frac12
 \tr(\begin{pmatrix}
     \Delta^T\Delta & 0\\
     0 & \Delta\Delta^T
 \end{pmatrix})=\tr(\Delta^T\Delta)=1.
 \end{equation*}
The first and last term of $d\varphi_B(\Delta)$ in \eqref{eq:varphi_diff} have the same norm.
\begin{align*}
    \norm{\begin{pmatrix}
        0&S\Delta^T\\
        0&BS\Delta^T
    \end{pmatrix}}_0^2 &= 
    \frac12 \trace \begin{pmatrix}
        0&0\\
        \Delta S&\Delta SB^T
    \end{pmatrix}
    \begin{pmatrix}
        0&S\Delta^T\\
        0&BS\Delta^T
    \end{pmatrix}\\
    &=\frac12\trace\begin{pmatrix}
        0& 0\\
0&\Delta S S\Delta^T +   \Delta SB^T BS\Delta^T\end{pmatrix}\\
    &= \frac12(\normF{S\Delta^T}^2 + \normF{BS\Delta^T}^2)\\
    &\leq \frac12\left( \norm{S}_2^2\normF{\Delta}^2 + \norm{BS}_2^2\normF{\Delta}^2\right)\\
    &= \frac12 \left(\frac{1}{(1+\sigma_{\min}^2)^2} + 
    \max_{i=1,\ldots p}\qty{ \frac{\sigma_i^2}{(1+\sigma_{i}^2)^2}}\right),
\end{align*}
where we have used that $\|XY\|_F\leq \|X\|_2\|Y\|_F$, see \cite[Lemma 3]{zimmermannstoyeSIMAX:2024}.
Since $g(t) = \frac{1}{(1+x^2)^2}\leq 1$ and $f(t)=\frac{t^2}{(1+t^2)^2}\leq \frac{1}{4}$, we have the global bound. 
\end{proof}
\subsection{From Euclidean distances to manifold distances}
\label{sec:Euclid2Grassmann_distance}
The Grassmann manifold $(\Gr(n,p), \langle \cdot ,\cdot\rangle_0)$ of orthogonal projectors is an embedded submanifold of $(\sym(n), \langle \cdot ,\cdot\rangle_0)$. Hence, \Cref{lem:Eulcid2EmbeddedMan_distance}
applies and yields comparative bounds between the intrinsic Riemannian distance and the extrinsic Euclidean distance.
To make the bounds concrete, we need the maximum Euclidean curvature of a unit-speed Grassmann geodesic when considered as a space curve in $(\sym(n), \langle \cdot ,\cdot\rangle_0)$.
\begin{lemma}
\label{lem:geo_curvature}
    Let $t\to y(t)$ be a unit--speed geodesic on $\left(\Gr(n,p),\inp{\cdot}{\cdot}_0\right)$.
    Let $\kappa(t)=\|\ddot{y}(t)\|$ denote the Euclidean curvature of $y$, when considered as a space curve in the ambient space $(\sym(n), \langle \cdot ,\cdot\rangle_0)$.
    Then $\kappa(t)\leq 2$.
\end{lemma}
\begin{proof}
    Consider a Stiefel representative $\gamma(t)=Q_0\exp_m \qty(t\begin{pmatrix}
        0&-\Delta^T\\
        \Delta&0
    \end{pmatrix})\begin{pmatrix}
        I_p\\
        0
    \end{pmatrix}$
    of a unit--speed geodesic $y(t)=\gamma(t)\gamma(t)^T$ on $\Gr(n,p)$. Without loss of generality, $Q_0=I_n$. With $X:=\begin{pmatrix}
        0&-\Delta^T\\
        \Delta&0
    \end{pmatrix}$ it holds
    \[
        \dot y(t)=\dot \gamma(t)\gamma(t)^T+\gamma(t)\dot\gamma(t)^T%\\
        %=&\exp_m\qty( t\Delta)\Delta\begin{bmatrix}
        %    I_p\\
        %    0
        %\end{bmatrix}\begin{bmatrix}
        %    I_p&0
        %\end{bmatrix}\exp_m\qty(- t\Delta)-%\exp_m\qty( t\Delta)\begin{bmatrix}
        %    I_p\\
        %    0
        %\end{bmatrix}\begin{bmatrix}
        %    I_p&0
        %\end{bmatrix}\Delta\exp_m\qty(t\Delta).\\
        =\exp_m\qty(tX)\qty(XP_0+P_0X)\exp_m\qty(-tX).
    \]
    Unit--speed implies 
    $
        1=\norm{\dot y(t)}_0^2=\norm{XP_0+P_0X}_0^2
        %=\norm{\begin{bmatrix}
        %    0&-B^T\\
        %&    B&0
        %\end{bmatrix}}_0^2
        =\frac{1}{2}\qty(2\tr(\Delta^T\Delta))=\normF{\Delta}^2.
    $

    This yields that $
        \ddot y(t)=\exp_m\qty(tX)\begin{pmatrix}
            2\Delta^T\Delta&0\\
            0&2\Delta\Delta^T
        \end{pmatrix}\exp_m\qty(-tX)$
    has norm
    \begin{align*}
       \kappa(t)^2= \norm{\ddot y(t)}_0^2=\frac{1}{2}\tr(\begin{pmatrix}
            4\Delta^T\Delta\Delta^T\Delta&0\\
            0&4\Delta\Delta^T\Delta\Delta^T
        \end{pmatrix})
        =4\|\Delta^T\Delta\|_F^2
        \leq 4.
    \end{align*}
\end{proof}
Any two points $P,\tilde P\in \Gr(n,p)$ with $\dist(P,\tilde P)<\frac{\pi}{2}$ can be connected by a unique minimizing 
geodesic \cite[Proposition 5.2]{Bendokat:2024}.
\Cref{lem:Eulcid2EmbeddedMan_distance} combined with \Cref{lem:geo_curvature} gives
\begin{equation}
    \dist(P,\tilde P) 
    \geq \|P-\tilde P\|_0\geq \sin(\dist(P,\tilde P)).
\end{equation}

Applying the mean--value inequality \cite{Hall:1979},  we obtain for $B_1,B_2\in \R^{(n-p)\times p}$ 
\begin{equation*}
    \norm{\varphi(B_1)-\varphi(B_2)}_0\leq \max_{t\in[0,1]}\norm{d\varphi_{(B_2+t(B_1-B_2))}}_{\textnormal{F},0} \normF{B_1-B_2}=M\normF{B_1-B_2},
\end{equation*}
with $M=\max_{t\in[0,1]}\norm{d\varphi_{(B_2+t(B_1-B_2))}}_{\textnormal{F},0}$.

With this bound at hand, we can estimate the manifold interpolation error in terms of the interpolation error in local coordinates by combining \Cref{lem:Eulcid2EmbeddedMan_distance} and \Cref{lem:cond_varphi}
\begin{theorem}\label{thm:interpolation_error}
    Let $B,\tilde B\in\R^{(n-p)\times p}$ two points in the coordinate space, which we interprete as true and interpolated data, respectively.
    Let $P=\varphi(B), \tilde P = \varphi(\tilde B)$ be the corresponding manifold images.
    Then 
    \begin{equation*}
        \dist(P,\tilde P)\leq \arcsin(M\normF{B-\tilde B})=%\norm{d \varphi_B}_{F,0}
        M\normF{B-\tilde B}+\mathcal{O}\qty(\normF{B-\tilde B}^3),
    \end{equation*}
    provided $\normF{B-\tilde B}<\frac{1}{M}$ and $\dist(P,\tilde P)<\frac{\pi}{2}$. The amplification factor $M$ 
    %is $\max_{t\in[0,1]}\norm{d\varphi_{(B_2+t(B_1-B_2))}}_{F,0}$ and 
    is addressed in \Cref{lem:cond_varphi}  and is at most $\sqrt{\frac52} +1$.
\end{theorem}
\begin{proof}
    The first inequality is a direct consequence of \Cref{lem:Eulcid2EmbeddedMan_distance} combined with the curvature result of \Cref{lem:geo_curvature}. A Taylor expansion of $\arcsin$ yields the second equality. 
\end{proof}
\begin{remark}
In particular, \Cref{thm:interpolation_error} shows that the order of the interpolation error is preserved when transiting from data on the coordinate domain to manifold data. 
Take for example the classical interpolation error estimates for Lagrange and Hermite, see e.g. \cite[Sections 8.1 and 8.5]{Alfio:2007}. We immediately obtain error bounds for Grassmann Lagrange-- and Hermite interpolation.
Consider the data $\qty{(t_1,f_1,f_1'),(t_2,f_2,f_2')}$ with $t_2-t_1=h\to 0$ and  $K_L,K_H>0$. Then
\begin{align*}
    \dist(f(t^*),\hat f_L(t^*))&\leq K_LMh^2+\mathcal{O}(h^6)=\mathcal{O}(h^2)\\
    \dist(f(t^*),\hat f_H(t^*))&\leq K_HMh^4+\mathcal{O}(h^{12})=\mathcal{O}(h^4),
\end{align*}
where $\hat f_L$ and $\hat f_H$ are  the Lagrange- and Hermite interpolant, respectively. 
\end{remark}
An illustration of the bound $\dist(P,\tilde P)\leq \arcsin(M\normF{B-\tilde B})$ is shown in \Cref{fig:theorem_8}, and we observe that it is not sharp. 
\begin{figure}[!ht]
    \centering
    \includegraphics[width=0.8\linewidth]{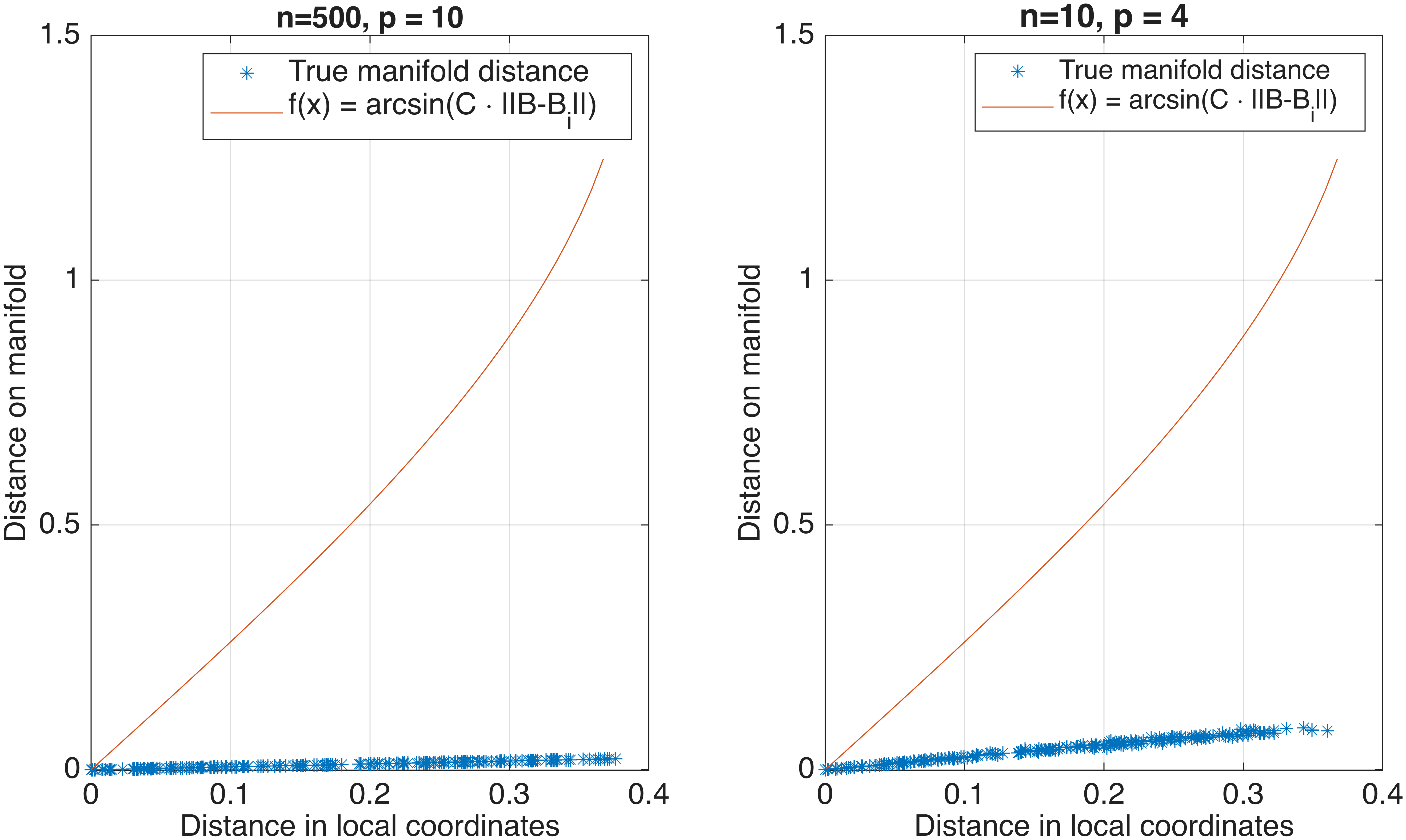}
    \caption{Illustration of the bound in \Cref{thm:interpolation_error} with $n=500, p=10$ and $n=10,p=4$. $C=\sqrt{\frac{5}{2}}+1$ comes from \eqref{eq:cond_varphi}. To produce the figure we mapped a point $P\in \Gr(n,p)$ to its MV coordinate $B$ and generated 200 random coordinates in a small neighborhood of $B$, $\qty{B_i}_{i=1}^{200}$. Subsequently we mapped these points to $\Gr(n,p)$ via $\varphi$ and computed the distance between $P$ and the Grassmann points $\qty{P_i}_{i=1}^{200}$ as $\dist(P,P_i)=\norm{\Log_P(P_i)}_0$.    }
    \label{fig:theorem_8}
\end{figure}

\section{Experimental results}
\label{sec:experiments} 
We will now present numerical evidence that interpolation in local coordinates can be competitive in terms of the interpolation accuracy when compared to interpolation in Riemannian normal coordinates. As indicated in the previous section we will focus on Lagrange-- and Hermite interpolation. 

In the experiments, we will interpolate data pairs exclusively, eg. two points and their derivative information. In this special case, Lagrange interpolation in local coordinates becomes linear interpolation and the normal coordinate variant becomes geodesic interpolation. 

\subsection{Interpolation of Q--factor projector of a matrix curve}

We consider the following academic example also considered in \cite{ZimmermannHermite:2020} and \cite{SeguinKressner:2024}: construct a matrix curve in $\R^{n\times p}$, with $n=1000,p=10$
\begin{equation*}\label{eq:exp_1_curve}
    Y(t)=Y_0+tY_1+t^2Y_2+t^3Y_3,
\end{equation*}
where $Y_0,Y_1,Y_2,Y_3\in \R^{n\times p}$ are randomly uniformly generated so that $Y_0$ has entries in $[0,1]$, $Y_1,Y_2$ have entries in $[0,0.5]$ and $Y_3$ has entries in $[0,0.25]$. We then consider the curve $U(t)=\textnormal{qf}(Y(t))$, where $\textnormal{qf}$ is the $Q$ factor of the QR--decomposition $U(t)R(t)=Y(t)$. The rank of $Y(t)$ is almost surely equal to $p=10$. We view the curve $t\mapsto U(t)$ as a Stiefel representation of a curve $U(t)U(t)^T$ on $\Gr(n,p)$. The derivative information $U'(t)$ is obtained by differentiating the QR--decomposition \cite[Proposition 2.2]{Walter:2012}. 

In \Cref{fig:exp_1} we present a numerical experiment, where we interpolate $P(t)=U(t)U(t)^T$ via the Stiefel representative $U(t)$ on $[0,1]$ with point and derivative data obtained at $t=0$ and $t=1$. The MV coordinates are the local coordinates multiplied with the permutation matrix $\mcP$ obtained from applying \Cref{alg:prepare_data} to the data point at $t=1$. The norms of the inverses of the upper $p\times p$ blocks before and after applying \Cref{alg:prepare_data} are shown in \Cref{tab:norms_exp1}.
As can be seen from the table, the data point at $t=0$ also experienced a reduction in condition number of its upper $p\times p$--block even though $\mcP$ was obtained from the maxvol procedure applied to $U(1)$.
\begin{table}[!h]
    \centering
    \begin{tabular}{|c|c|c|}\hline
         & Before \cref{alg:prepare_data}  & After \cref{alg:prepare_data} \\\hline
         $\normF{U_1^{-1}(0)}$ & $87.3$&$38.6$\\
         $\normF{U_1^{-1}(1)}$ &$142.1$ & $29.9$\\ \hline
    \end{tabular}
    \caption{The norm of the inverses of the upper $p\times p$ blocks of the data obtained by evaluating the curve $U(t)$ at $t=0$ and $t=1$.}
    \label{tab:norms_exp1}
\end{table}

When permuting the columns of the point data to prepare for computing the MV coordinates, one must apply the same permutation on the derivative data before mapping it to the local coordinate domain using \eqref{eq:d_psi}.  After mapping the point and derivative data to the local coordinate domain, we compute the Lagrange-- and Hermite interpolant at $t^*\in [0,1]$ to obtain $\bar P(t^*)\in \R^{(n-p)\times p}$. If $\bar P(t^*)$ is obtained from MV coordinate data, then the manifold interpolant is given as $\hat P(t^*)=\mcP^T\varphi(\bar P(t^*))$, where the parametrization $\varphi$ is computed by via a Cholesky decomposition (see \Cref{sec:numprocessing}). If the interpolant is obtained from standard local coordinates, then there is no application of $\mcP^T$. 

\begin{figure}[!ht]
    \centering
    \includegraphics[width=0.8\linewidth]{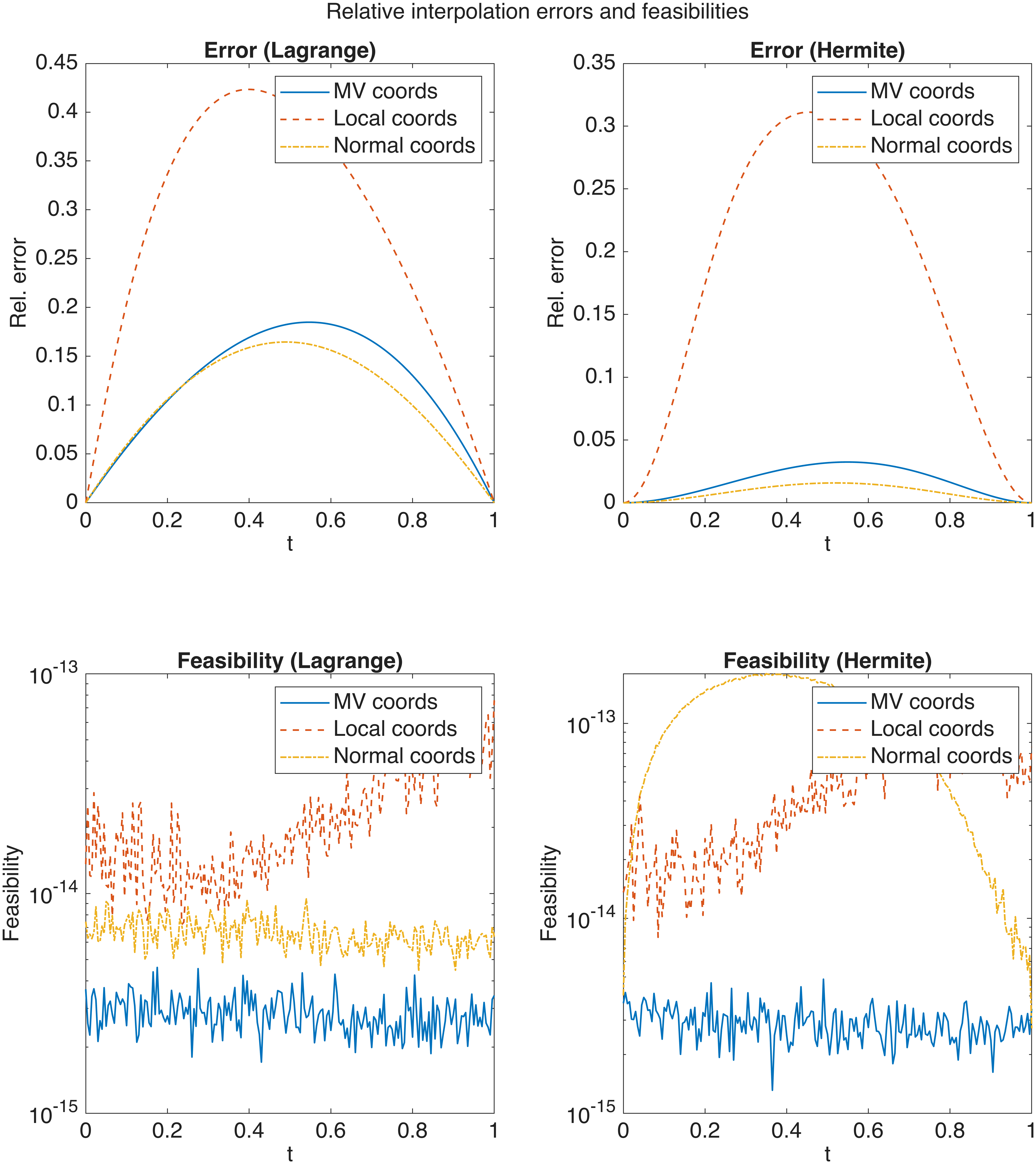}
    \caption{Relative errors and feasibilities for the six interpolation schemes when interpolating the projector $P(t)=U(t)U(t)^T$, with $U(t)$ being the Q--factor of the QR--decomposition of $Y(t)$ in \eqref{eq:exp_1_curve}. The relative error is computed as $\frac{\normF{P(t^*)-\hat P(t^*)}}{\normF{P(t^*)}}$ and feasibility refers to the quantity $\normF{\hat U(t^*)^T\hat U(t^*)-I_p}$, where $\hat U(t^*)$ is the obtained Stiefel representative of $\hat P(t^*)=\qty[\hat U(t^*)]$. 
    }
    \label{fig:exp_1}
\end{figure}

As expected, we observe that interpolation in Riemannian normal coordinates results in smaller interpolation errors compared to the approach using local coordinates. 
However, the difference to the errors associated with the MV-coordinates is not large, see \Cref{fig:exp_1}.
Moreover, the two lower plots in this figure show that the interpolation approach in MV-coordinates is best in preserving the numerical orthogonality of the associated Stiefel matrix representatives.
Despite the fact that the data did not feature ill-conditioned upper $p\times p$ blocks, we observe a significant reduction of the interpolation error when employing MV coordinates rather than the standard local coordinates.
This suggests that the matrix-decomposition free Grassmann coordinates should only be used in combination with the volume maximization.

\subsection{Parametric model reduction of a dynamical system}
\label{sec:FN-model}
In computational neuroscience, one is interested in modeling the electronic dynamics of neurons, which is applicable to, for example, brain science and prediction of experiments. The FitzHugh--Nagumo model (FN--model) \cite[Section 1.3.3]{Mallot:2024}, \cite{Lacasa:2024}, a two dimensional simplification of the four dimensional Hodgkin--Huxley model, models neuron activity. Here it is formulated with a diffusion term in one spatial variable. We follow the model set-up of \cite[Section 4.1]{Chaturantabut:2010}; however this refernce does not consider parametric interpolation of reduced-order models.
The governing equations are
\begin{subequations}\label{eq:FN}
    \begin{align}
        \frac{\partial u(t,x)}{\partial t}&=\varepsilon\frac{\partial^2 u(t,x)}{\partial x^2}+\frac{1}{\varepsilon}f(u(t,x))-\frac{1}{\varepsilon}v(t,x)+I_a,\label{eq:FNa}\\
        \frac{\partial v(t,x)}{\partial t}&=bu(t,x)-\gamma v(t,x) \label{eq:FNb}. 
    \end{align}
\end{subequations}
Here $f(u)=u(u+1)(1-u)$ is a nonlinear function, $b,\gamma,\varepsilon$ are constants, and $I_a$ represent externally applied voltage. The boundary conditions are 
\begin{equation*}
    \textnormal{Boundary conditons}=\begin{cases}
        u(0,x)=0,v(0,x)=0,\\
        u_x(t,0)=-50000t^3e^{-15t}=\beta(t),\\
        u_x(t,L)=0.
    \end{cases}
\end{equation*}

In this experiment we will interpolate low-rank subspaces associated with model reduction of the FN--system \eqref{eq:FN} by \textit{proper orthogonal decomposition (POD)}. The idea is as follows. Obtaining numerical realizations of dynamical systems can be computationally costly. To efficiently obtain solutions, snapshot--based model order reduction can be employed. A \textit{snapshot} is a state vector of the dynamical system computed with a numerical method at a specific time instance. A snapshot matrix is a collection of such snapshots. For the FN-system, the externally applied voltage $I_a$ is a system parameter. Assume that the interval $[\alpha,\beta]$ is the parameter domain that is of interest for $I_a$. We choose voltage sample locations $I_a\in [\alpha,\beta]$, and for each $I_a$, we compute a time trajectory for the FN-system and collect snapshots in $I_a$-dependent snapshot matrices $Y_{I_a}$. Subsequently, for each sampled $I_a$, we compute a POD basis from the SVD of $Y_{I_a}$, and store the dominant $p$ left--singular vectors $\Phi_{I_a}$, such that $\colspan(\Phi_{I_a})\approx \colspan(Y_{I_a})$. With this set of sample data at hand, we can approximate POD subspaces for each untried $I_a\in [\alpha,\beta]$ by Grassmann--interpolating the sampled POD bases, or equivalently by interpolating the POD subspaces.  

We note that plain--vanilla POD--based model order reduction is not sufficient for building an efficient surrogate model for \eqref{eq:FN}, as the nonlinear term affects the computational cost. To alleviate this, one can use the DEIM method presented in \cite{Chaturantabut:2010}. We chose not to include this here, since our focus is on the quality of the subspace interpolation. 

Using discretization as in \cite[Chapter 4.1]{Chaturantabut:2011}, we compute snapshot matrices at $6$ sample locations $I_a=0.03,0.04,\dots,0.08$, with the other system parameters kept fixed at $L=1, T=8, \epsilon=0.015, b=0.3$ and $\gamma=0.5$. Each snapshot matrix consists of $1001$ time samples of the FN--system, taken at $t\in [0,8]$. For each snapshot matrix, we compute the dominant $p$ left singular vectors of the $u$ resp. $v$ part of the state vector snapshots, so that we obtain low--rank POD bases $U$ and $V$ for $u$ resp. $v$ at each of the sampled parameters $I_a$. The dimensions are $n=1024$ and $p=8$, where the latter is chosen according to the singular value drop of the snapshot matrices and the former $n=1024$ comes from the resolution of the FD approximation of the second-order term. By approximating the derivatives $\dv{I_a}u(t)$ and $\dv{I_a}v(t)$ via finite differences, we can obtain the derivatives of the respective POD bases at each $I_a$ by differentiating the SVD used to compute the POD basis \cite[Section S3]{ZimmermannHermite:2020}. Doing so for each $I_a$, we obtain the sample data (all in rectangular Stiefel representations) 
\begin{equation*}
    \mcS_U = \qty{I_a,U_{I_a},U'_{I_a}\mid I_a = 0.03,\ldots, 0.08},  \ \ \mcS_V = \qty{I_a,V_{I_a},V'_{I_a}\mid I_a = 0.03,\ldots, 0.08}.
\end{equation*}
Restricting to $\mcS_U$, we construct an interpolant for each pair of consecutive data points.
The procedure for computing the interpolants is the same as in the previous experiment. 
In \Cref{tab:norms_exp2}, we present the results of the volume maximization of \Cref{alg:prepare_data}.
The table shows that the matrix blocks associated with the upper $p$ entries of the state vectors are not at all suitable for interpolation and volume maximization is mandatory.
We present the relative interpolation errors in \Cref{fig:exp_2} for $I_a\in[0.03,0.08]$. 
\begin{table}[h]
    \centering
    \begin{tabular}{|c|c c| c c|}\hline
        & \multicolumn{2}{c}{Before \Cref{alg:prepare_data}}& \multicolumn{2}{c|}{After \Cref{alg:prepare_data}} \\
        {\small Interval} & {\small Left} & {\small Right} & {\small Left} & {\small Right}\\ \hline
        $[0.03,0.04]$ &$1.7\cdot 10^{14}$ &$1.8\cdot 10^{13}$ &$30.5$ &$30.6$ \\
        $[0.04,0.05]$ &$1.8\cdot 10^{13}$ &$1.3\cdot 10^{13}$ &$30.8$ &$31.1$ \\
        $[0.05,0.06]$ &$1.3\cdot 10^{13}$ &$1.4\cdot 10^{13}$ &$30.8$ &$30.8$ \\
        $[0.06,0.07]$ &$1.4\cdot 10^{13}$ &$1.1\cdot 10^{13}$ &$30.6$ &$30.7$ \\
        $[0.07,0.08]$ &$1.1\cdot 10^{13}$ &$9.5\cdot 10^{12}$ &$30.6$ &$30.6$ \\\hline
    \end{tabular}
    \caption{Norms of the inverse upper $p\times p$ blocks of the POD basis matrices associated with the FN--system \eqref{eq:FN} before and after applying the volume-maximizer \Cref{alg:prepare_data}. `Left' and `right' refer to the data matrices at the lower and upper bound of the interval in question. For example when interpolating on $[0.03,0.04]$, `left' is the point $U(0.03)$ and `right'  is $U(0.04)$. In the case at hand, significant reductions by 12 orders of magnitude are observed.}
    \label{tab:norms_exp2}
\end{table}

In line with the previous experiment, we observe that interpolation in MV coordinates and Riemannian normal coordinates leads to similar errors. There are local regions where the MV-interpolation even outperforms interpolation in normal coordinates.\\
We also want to juxtapose the interpolation errors in local coordinates before and after the volume maximization procedure. Yet, with the data as they are, the problem is so ill-conditioned that the interpolation is completely numerically corrupted without any volume maximization.
Hence, to be able to visualize the effects, we ran the maxvol algorithm and terminated after only 3 iterations.
This was sufficient for all routines to become operational. 
In \Cref{fig:exp_2_2} we display the associated interpolation errors. 
As can be seen from this figure, without proper volume maximization, the interpolation errors are generally large and occasionally feature strong spikes.
 \begin{figure}[!ht]
    \centering
    \includegraphics[width=0.8\linewidth]{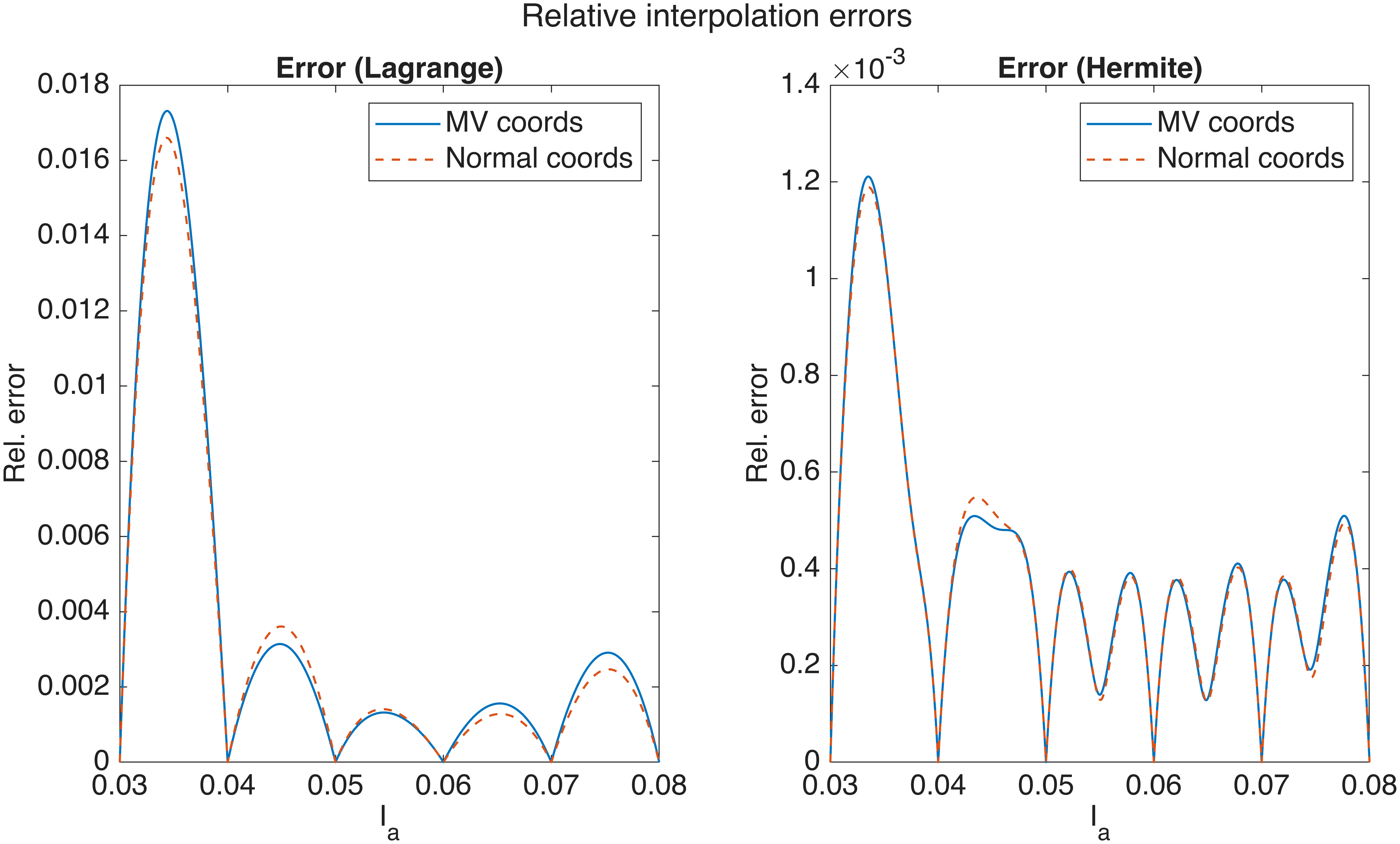}
    \caption{Relative errors for the four interpolation schemes when interpolating the POD bases of the FN model for varying $I_a$.}
    \label{fig:exp_2}
\end{figure}
\begin{figure}
    \centering
    \includegraphics[width=0.8\linewidth]{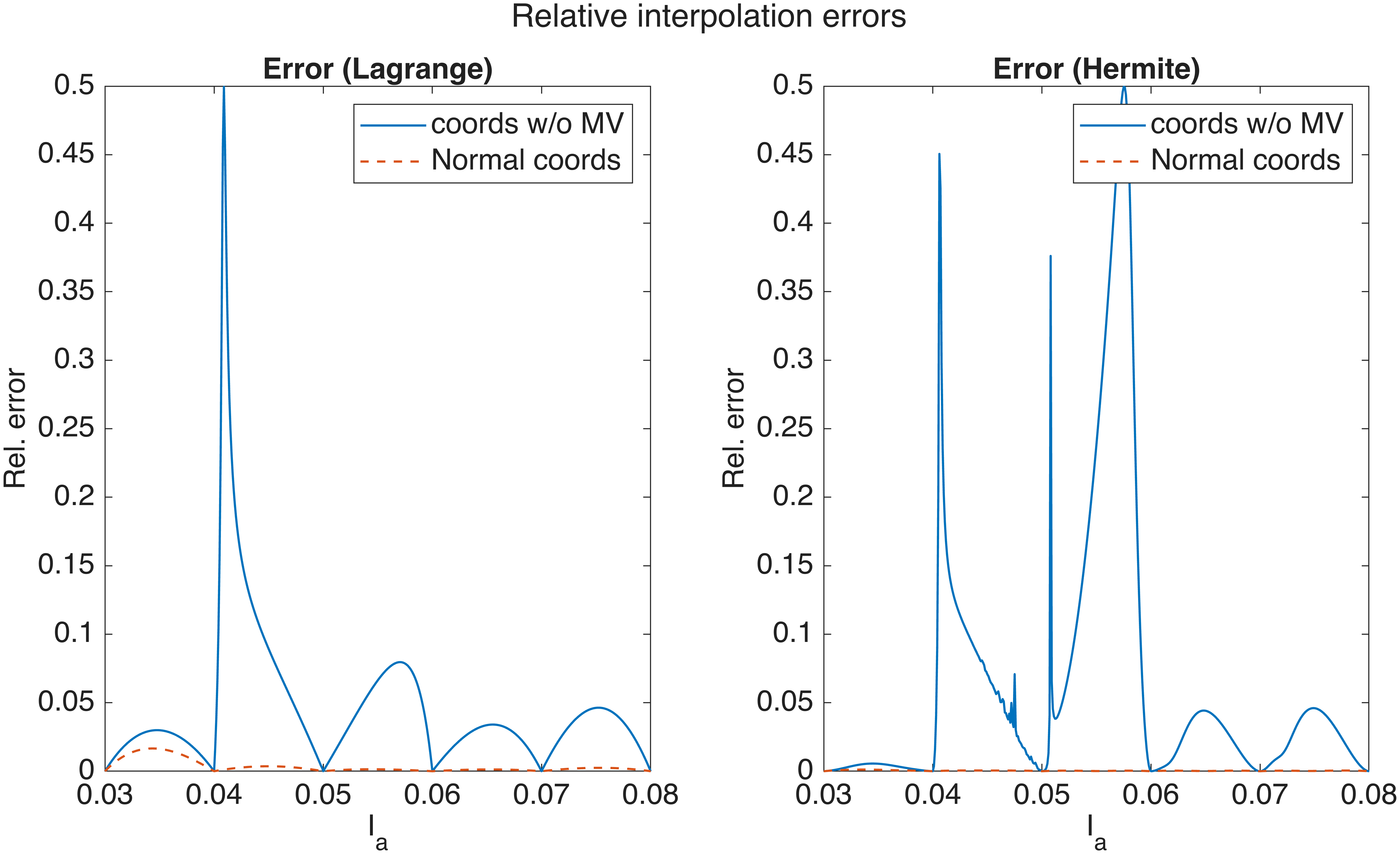}
    \caption{Interpolating the POD subspaces of the FN model: Riemannian normal coordinates vs. local coordinates without proper volume maximization. The experiment shows that volume maximization should not be omitted.}
    \label{fig:exp_2_2}
\end{figure}

\section{Summary and outlook}
\label{sec:conclusions}

We have presented a method for interpolation in local coordinates on the Grassmann manifold, which, as opposed to working in Riemannian normal coordinates and many other common retractions, does not require any matrix decompositions. For the Grassmann manifold, we have devised a strategy for obtaining maximum--volume coordinates. Employing MV coordinates has proven to be highly beneficial for the interpolation accuracy or even strictly necessary, in order to avoid algorithmic breakdowns.

We have conducted a thorough error analysis and quanitfied how errors associated with (Euclidean) interpolation in the local coordinate domain are forward to interpolation errors on the Grassmann manifold. In particular, the error analysis shows that the asymptotic interpolation error is preserved. This means that if the interpolation method in coordinate domain has an asymptotic error $\mathcal{O}(h^k)$, for step-size $h\to 0$, then same holds for the interpolation error on the manifold. 

The focus of this paper is on a comparison of the performance of Riemannian interpolation on $\Gr(n,p)$ under different coordinates, not on the interpolation as such.
For simplicity, Lagrange-- and Hermite interpolation between only two points was considered. However, it is straightforward to construct higher--order interpolants using data from more than two points, provided that all data can be mapped to the same coordinate domain. 
For applications in parametric model-order reduction, it seems reasonable to expect that the dominant information is contained in similar locations across a data set of POD basis matrices, so that maximum--volume coordinates that provide a well-conditioned sub-block for a selected sample matrix, are unlikely to lead to ill-conditioned sub-blocks at the other sample matrices in a given data set.
How to truly generalize the maximum--volume sub-block search to a data set of multiple matrices is an open question.

In the numerical experiments, we have observed that the error that is introduced at the preprocessing stage due to the conditioning of the coordinate charts has a considerable impact on the final interpolation accuracy. 
When utilizing the optimized  maximum-volume coordinates, the interpolation errors in the numerical experiments were on par with those associated with the Riemannian normal coordinates.

A task for future research is to quantify the conditioning (and associated preprocessing errors) for other Grassmann coordinate charts such as the inverses of common Grassmann retractions.

\begin{center}
   \begin{tikzpicture}[scale = 2.1]
   \duck[body=yellow!50!brown!40!white,
       crazyhair=gray!50!white,
       eyebrow,
       glasses=brown!70!black,
       sunglasses=white,
       book=\scalebox{0.6}{
\parbox{2cm}{\color{black}
\centering Riem.\\ Geo.}},
       bookcolour=red!20!brown,
       laughing,
       tshirt=white,
       jacket=black,
        vampire=white,
        signpost={$\Gr(n,p)$}
        ]
\begin{scope}
   \node[inner sep=0pt] at (0.75,0.6)
   {\includegraphics[width=.08\textwidth,angle=-28,origin=c]{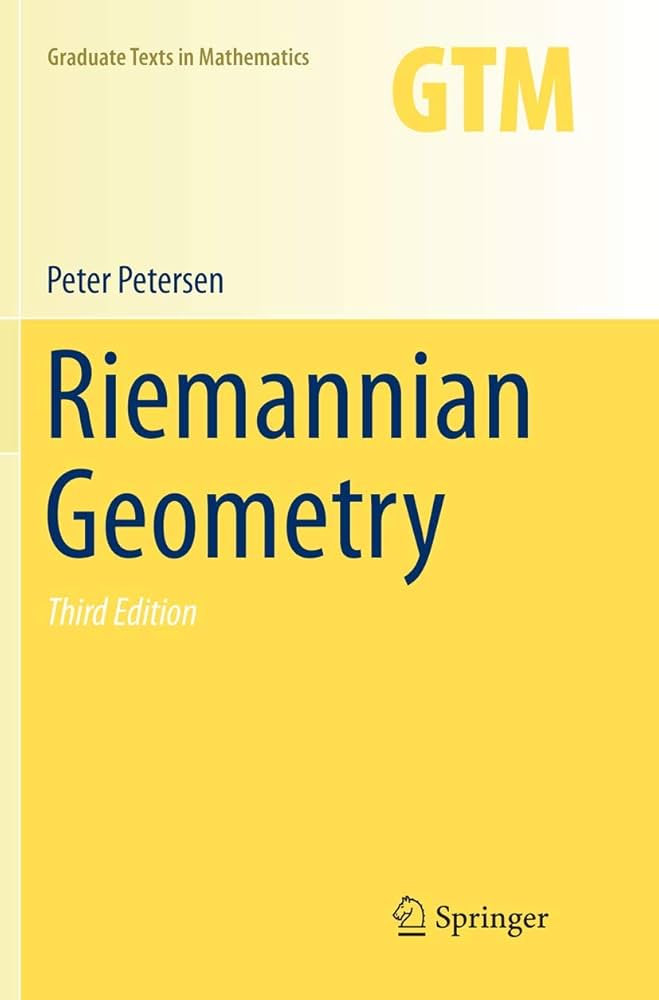}};
\end{scope}
\end{tikzpicture}
\end{center}

\backmatter

\section*{Declarations}

%Some journals require declarations to be submitted in a standardised format. Please check the Instructions for Authors of the journal to which you are submitting to see if you need to complete this section. If yes, your manuscript must contain the following sections under the heading `Declarations':

\begin{itemize}
\item This work was supported by the Independent Research Foundation Denmark, DFF, grant nr. 3103-00094B.
\item The authors declare no conflicts of interest relevant for the present work. 
%\item Ethics approval and consent to participate
%\item Consent for publication
\item The source code for the numerical experiments is available at \url{https://github.com/RalfZimmermannSDU/GrassmannInterpolation.git}.
\end{itemize}

\noindent

\begin{appendices}

\section{Riemannian normal coordinates}\label{app:rie_norm_coords}
The Riemannian exponential map with domain on $\T_P\Gr(n,p)$ sends a tangent vector $X$ to the endpoint of the geodesic $c_{P,X}:[0,1] \to \Gr(n,p)$ that starts from $P$ with velocity $X$. Hence, the exponential provides manifold images for tangent vectors
\[
    \Exp_P: \T_P\Gr(n,p) \to \Gr(n,p), \quad X \mapsto \tilde P =c_{P,X}(1).
\]
In general, the Riemannian exponential is not globally defined on the asssociated tangent spaces, but on the Grassmann and Stiefel manifolds it is. These manifolds are goedesicallly complete, thus complete, see \cite[Hopf-Rinow's Theorem, Chap. 7, Thm 2.8]{DoCarmo2013riemannian}.
The Riemannian exponential is locally a diffeomorphism. The inverse map is called the Riemannian logarithm
\[
    \Log_P: \Gr(n,p) \supset \mathcal{D} \to \T_P\Gr(n,p), \quad \tilde P \mapsto X = (\Exp_P)^{-1}(\tilde P).
\]
Matrix formulae for both maps are in \cite{EdelmanAriasSmith:1999, AbsilMahonySepulchre2004, Bendokat:2024}. The collection of Riemannian logarithm maps provides an atlas for the Grassmann manifold, i.e., a collection of coordinate charts. The Riemannian exponential maps are the associated parameterizations. Together, they form the {\em Riemannian normal coordinates}.
The Riemannian normal coordinates are radially isometric in the sense that the 
Riemannian distance between $P$ and $\tilde P = \Exp_P(X)$ is exactly 
the same as the length of the tangent vector $\|X\|_0$, \cite[Lem. 5.10 \& Cor. 6.11]{Lee1997riemannian}.
For interpolation purposes, {\em isometric} coordinates would be ideal, because then the coordinate maps would be free of geometric distortions. However, there can be no full isometry between a flat tangent space and a curved manifold. Therefore, the Riemannian normal coordinates, which are at least radially isometric, can be considered the gold standard for data interpolation on manifolds.

\section{Interpolation in Riemannian normal coordinates}
\label{sec:Int_in_normal_coords}

In the following, when considering Lagrange interpolation of a dataset $\qty{t_i,f_i}_{i=1}^d$, we build the $d-1$ Lagrange polynomial, and compute the interpolant $\hat f(t^*)$ for $t^*\in [t_1,t_d]$. For Hermite interpolation, given $t^*\in [t_i,t_{i+1}]$, compute the Hermite interpolant based on the data $(t_i,f_i,f_i'), (t_{i+1},f_{i+1},f_{i+1}')$. 

\subsection{Lagrange interpolation in normal coordinates}

Recall that the Lagrange polynomial of order $d-1$ passing through data in a vector space $\qty{(t_i,g_i)}_{i=1}^d\subset \R\times \R^n$ is given by 
\begin{equation}\label{eq:lagrange_poly}
    L(t)=\sum_{i=1}^d g_i\prod_{\stackrel{j=1}{j\neq i}}^d \frac{t-t_i}{t_i-t_j}.
\end{equation}
Given a dataset on the Grassmann manifold and some $t^*\in [t_1,t_d]$  $\qty{t_i,f_i}_{i=1}^d$, fix $f_i$ and map the data to the tangent space $\T_{f_i}\Gr(n,p)$ using the Riemannian logarithm $\Delta_j=\Log_{f_i}(f_j)$ for each $j$. Letting $\Delta_j$ play the role of $g_j$ in \eqref{eq:lagrange_poly}, one can compute the Lagrange interpolant in the tangent space to obtain $L(t^*)$. The Lagrange interpolant on $\Gr(n,p)$ is obtained as $\hat f(t^*)=\Exp_{f_i}(L(t^*))$.

\subsection{Hermite interpolation in normal coordinates} 

Given a dataset containing point and derivative information $(t_0,f_0,f_0'),(t_1,f_1,f_1')$, we conduct Hermite interpolation in Riemannian normal coordinates, as it was formulated in general and for the Stiefel manifold in \cite{ZimmermannHermite:2020}. The task is to construct a curve in the tangent space with base at either $f_0$ or $f_1 $ which, when composed with the Riemannian exponential map, is the Hermite interpolant. We select the second data point $f_1$ as the anchor. It follows from \cite[Theorem 1]{ZimmermannHermite:2020} that the Hermite interpolation curve in the tangent space $\T_{f_1}\Gr(n,p)$ is given by 
\begin{equation*}
    \mu(t)=L_{00}(t)\xi_{f_0}+L_{10}(t)\cdot 0+L_{01}(t)v_{f_0}+L_{11}(t)v_{f_1},
\end{equation*}
where the Hermite coefficient functions are \cite[Section 8.5]{Alfio:2007}
\begin{align*}
    L_{00}(t)&=\qty(1-\frac{2}{t_0-t_1}(t-t_0))\qty(\frac{t-t_1}{t_0-t_1}^2),\\
    L_{10}(t)&=\qty(1-\frac{2}{t_1-t_0}(t-t_1))\qty(\frac{t-t_0}{t_1-t_0})^2,\\
    L_{01}(t)&=(t-t_0)\qty(\frac{t-t_1}{t_0-t_1})^2,\\
    L_{11}(t)&=(t-t_1)\qty(\frac{t-t_0}{t_0-t_1})^2.
\end{align*}
Note that $\mu(t_1)=0$ ensures that $\Exp_{f_1}(\mu(t_1))=f_1$, 
$\xi_{f_0}=\Log_{f_1}(f_0)$, $v_{f_{1}}=f_1'$ and $v_{f_0}$ is given by 
\begin{align*}
    v_{f_0}&=\left.\deriv[s]\right|_{s=0}(\Log_{f_1}\circ \Exp_{f_0})(sf_0')\\
    &=\frac{(\Log_{f_1}\circ \Exp_{f_0})(hf_0')-(\Log_{f_1}\circ \Exp_{f_0})(-hf_0')}{2h}+\mathcal{O}(h^2).
\end{align*}
Finally, the Hermite interpolant is given by $\hat f(t^*)=\Exp_{f_1}(\mu(t^*))$.

\section{Grassmann subspace distance}
\label{app:subspace_dist}
Every two subspaces $\mathcal{U},\tilde{\mathcal{U}}\in \Gr(n,p)$ can be connected by a shortest (but not necessarily unique) geodesic, \cite{Wong1967}, \cite[Section 5]{Bendokat:2024}. The Riemannian distance between two subspaces is the length of such a shortest connecting geodesic and 
can be computed via the principal angles between $\mathcal{U}$ and $\tilde{\mathcal{U}}$.

If $U,\tilde U\in\St(n,p)$ are column-orthogonal matrices with 
$\colspan(U)=\mathcal{U}, \colspan(\tilde U) = \tilde{\mathcal{U}}$, then the principal angles are
$\theta_k := \arccos(\sigma_k) \in [0,\frac{\pi}{2}]$, 
where $\sigma_k$ is the $k$-th singular value of $U^T\tilde{U}\in\R^{p\times p}$.
Let $\Theta = (\theta_1,\ldots,\theta_p)^T\in[0,\frac{\pi}{2}]^p$ be the vector of principal angles.
Then, the Riemannian subspace distance is
 \begin{equation}
 \label{eq:subspace_dist}
  \dist(\mathcal{U},\tilde{\mathcal{U}}) = \|\Theta\|_2= \sqrt{\sum_k \theta_k^2}\leq \sqrt{p}\frac{\pi}{2}.
 \end{equation}

\end{appendices}


\bibliography{sn-bibliography}% common bib file


\begin{thebibliography}{62}
\ifx \bisbn   \undefined \def \bisbn  #1{ISBN #1}\fi
\ifx \binits  \undefined \def \binits#1{#1}\fi
\ifx \bauthor  \undefined \def \bauthor#1{#1}\fi
\ifx \batitle  \undefined \def \batitle#1{#1}\fi
\ifx \bjtitle  \undefined \def \bjtitle#1{#1}\fi
\ifx \bvolume  \undefined \def \bvolume#1{\textbf{#1}}\fi
\ifx \byear  \undefined \def \byear#1{#1}\fi
\ifx \bissue  \undefined \def \bissue#1{#1}\fi
\ifx \bfpage  \undefined \def \bfpage#1{#1}\fi
\ifx \blpage  \undefined \def \blpage #1{#1}\fi
\ifx \burl  \undefined \def \burl#1{\textsf{#1}}\fi
\ifx \doiurl  \undefined \def \doiurl#1{\url{https://doi.org/#1}}\fi
\ifx \betal  \undefined \def \betal{\textit{et al.}}\fi
\ifx \binstitute  \undefined \def \binstitute#1{#1}\fi
\ifx \binstitutionaled  \undefined \def \binstitutionaled#1{#1}\fi
\ifx \bctitle  \undefined \def \bctitle#1{#1}\fi
\ifx \beditor  \undefined \def \beditor#1{#1}\fi
\ifx \bpublisher  \undefined \def \bpublisher#1{#1}\fi
\ifx \bbtitle  \undefined \def \bbtitle#1{#1}\fi
\ifx \bedition  \undefined \def \bedition#1{#1}\fi
\ifx \bseriesno  \undefined \def \bseriesno#1{#1}\fi
\ifx \blocation  \undefined \def \blocation#1{#1}\fi
\ifx \bsertitle  \undefined \def \bsertitle#1{#1}\fi
\ifx \bsnm \undefined \def \bsnm#1{#1}\fi
\ifx \bsuffix \undefined \def \bsuffix#1{#1}\fi
\ifx \bparticle \undefined \def \bparticle#1{#1}\fi
\ifx \barticle \undefined \def \barticle#1{#1}\fi
\bibcommenthead
\ifx \bconfdate \undefined \def \bconfdate #1{#1}\fi
\ifx \botherref \undefined \def \botherref #1{#1}\fi
\ifx \url \undefined \def \url#1{\textsf{#1}}\fi
\ifx \bchapter \undefined \def \bchapter#1{#1}\fi
\ifx \bbook \undefined \def \bbook#1{#1}\fi
\ifx \bcomment \undefined \def \bcomment#1{#1}\fi
\ifx \oauthor \undefined \def \oauthor#1{#1}\fi
\ifx \citeauthoryear \undefined \def \citeauthoryear#1{#1}\fi
\ifx \endbibitem  \undefined \def \endbibitem {}\fi
\ifx \bconflocation  \undefined \def \bconflocation#1{#1}\fi
\ifx \arxivurl  \undefined \def \arxivurl#1{\textsf{#1}}\fi
\csname PreBibitemsHook\endcsname

%%% 1
\bibitem[\protect\citeauthoryear{Absil et~al.}{2004}]{AbsilMahonySepulchre2004}
\begin{barticle}
\bauthor{\bsnm{Absil}, \binits{P.-A.}},
\bauthor{\bsnm{Mahony}, \binits{R.}},
\bauthor{\bsnm{Sepulchre}, \binits{R.}}:
\batitle{{R}iemannian geometry of {G}rassmann manifolds with a view on algorithmic computation}.
\bjtitle{Acta Applicandae Mathematica}
\bvolume{80}(\bissue{2}),
\bfpage{199}--\blpage{220}
(\byear{2004})
\doiurl{10.1023/B:ACAP.0000013855.14971.91}
\end{barticle}
\endbibitem

%%% 2
\bibitem[\protect\citeauthoryear{Absil et~al.}{2008}]{Absilbook:2008}
\begin{bbook}
\bauthor{\bsnm{Absil}, \binits{P.-A.}},
\bauthor{\bsnm{Mahony}, \binits{R.}},
\bauthor{\bsnm{Sepulchre}, \binits{R.}}:
\bbtitle{Optimization Algorithms on Matrix Manifolds},
p. \bfpage{224}.
\bpublisher{Princeton University Press},
\blocation{Princeton, NJ}
(\byear{2008})
\end{bbook}
\endbibitem

%%% 3
\bibitem[\protect\citeauthoryear{Edelman et~al.}{1998}]{EdelmanAriasSmith:1999}
\begin{barticle}
\bauthor{\bsnm{Edelman}, \binits{A.}},
\bauthor{\bsnm{Arias}, \binits{T.A.}},
\bauthor{\bsnm{Smith}, \binits{S.T.}}:
\batitle{The geometry of algorithms with orthogonality constraints}.
\bjtitle{SIAM Journal on Matrix Analysis and Applications}
\bvolume{20}(\bissue{2}),
\bfpage{303}--\blpage{353}
(\byear{1998})
\doiurl{10.1137/S0895479895290954}
\end{barticle}
\endbibitem

%%% 4
\bibitem[\protect\citeauthoryear{Gallivan et~al.}{2003}]{Gallivan_etal2003}
\begin{bchapter}
\bauthor{\bsnm{Gallivan}, \binits{K.A.}},
\bauthor{\bsnm{Srivastava}, \binits{A.}},
\bauthor{\bsnm{Liu}, \binits{X.}},
\bauthor{\bsnm{Van~Dooren}, \binits{P.}}:
\bctitle{Efficient algorithms for inferences on {G}rassmann manifolds}.
In: \bbtitle{IEEE Workshop on Statistical Signal Processing},
pp. \bfpage{315}--\blpage{318}
(\byear{2003}).
\doiurl{10.1109/SSP.2003.1289408}
\end{bchapter}
\endbibitem

%%% 5
\bibitem[\protect\citeauthoryear{Man~Lui}{2012}]{Lui2012}
\begin{barticle}
\bauthor{\bsnm{Man~Lui}, \binits{Y.}}:
\batitle{Advances in matrix manifolds for computer vision}.
\bjtitle{Image and Vision Computing}
\bvolume{30}(\bissue{6--7}),
\bfpage{380}--\blpage{388}
(\byear{2012})
\doiurl{10.1016/j.imavis.2011.08.002}
\end{barticle}
\endbibitem

%%% 6
\bibitem[\protect\citeauthoryear{Sato}{2021}]{sato2021}
\begin{bbook}
\bauthor{\bsnm{Sato}, \binits{H.}}:
\bbtitle{{R}iemannian Optimization and Its Applications}.
\bsertitle{SpringerBriefs in Electrical and Computer Engineering}.
\bpublisher{Springer}, \blocation{???}
(\byear{2021})
\end{bbook}
\endbibitem

%%% 7
\bibitem[\protect\citeauthoryear{Boumal}{2023}]{boumal2023}
\begin{bbook}
\bauthor{\bsnm{Boumal}, \binits{N.}}:
\bbtitle{An Introduction to Optimization on Smooth Manifolds}.
\bpublisher{Cambridge University Press},
\blocation{Cambridge}
(\byear{2023})
\end{bbook}
\endbibitem

%%% 8
\bibitem[\protect\citeauthoryear{Minh and Murino}{2016}]{Minh:2016:AAR:3029338}
\begin{bbook}
\bauthor{\bsnm{Minh}, \binits{H.Q.}},
\bauthor{\bsnm{Murino}, \binits{V.}}:
\bbtitle{Algorithmic Advances in {R}iemannian Geometry and Applications: For Machine Learning, Computer Vision, Statistics, and Optimization}.
\bsertitle{Advances in Computer Vision and Pattern Recognition}.
\bpublisher{Springer},
\blocation{Cham}
(\byear{2016})
\end{bbook}
\endbibitem

%%% 9
\bibitem[\protect\citeauthoryear{Srivastava and Turaga}{2015}]{RiemannInComputerVision}
\begin{bbook}
\bauthor{\bsnm{Srivastava}, \binits{A.}},
\bauthor{\bsnm{Turaga}, \binits{P.K.}}:
\bbtitle{{R}iemannian Computing in Computer Vision}.
\bpublisher{Springer}, \blocation{???}
(\byear{2015}).
\doiurl{10.1007/978-3-319-22957-7}
\end{bbook}
\endbibitem

%%% 10
\bibitem[\protect\citeauthoryear{Pennec et~al.}{2020}]{Fletcher:2020}
\begin{bbook}
\bauthor{\bsnm{Pennec}, \binits{X.}},
\bauthor{\bsnm{Sommer}, \binits{S.H.}},
\bauthor{\bsnm{Fletcher}, \binits{T.}}:
\bbtitle{Riemannian Geometric Statistics in Medical Image Analysis},
\bedition{1.} edn.
\bpublisher{Academic Press},
\blocation{United States}
(\byear{2020}).
\doiurl{10.1016/C2017-0-01561-6}
\end{bbook}
\endbibitem

%%% 11
\bibitem[\protect\citeauthoryear{Benner et~al.}{2015}]{BennerGugercinWillcox2015}
\begin{barticle}
\bauthor{\bsnm{Benner}, \binits{P.}},
\bauthor{\bsnm{Gugercin}, \binits{S.}},
\bauthor{\bsnm{Willcox}, \binits{K.}}:
\batitle{A survey of projection-based model reduction methods for parametric dynamical systems}.
\bjtitle{SIAM Review}
\bvolume{57}(\bissue{4}),
\bfpage{483}--\blpage{531}
(\byear{2015})
\doiurl{10.1137/130932715}
\end{barticle}
\endbibitem

%%% 12
\bibitem[\protect\citeauthoryear{Celledoni et~al.}{2020}]{Celledoni_2020}
\begin{botherref}
\oauthor{\bsnm{Celledoni}, \binits{E.}},
\oauthor{\bsnm{Eidnes}, \binits{S.}},
\oauthor{\bsnm{Owren}, \binits{B.}},
\oauthor{\bsnm{Ringholm}, \binits{T.}}:
Energy preserving methods on {R}iemannian manifolds.
Mathematics of Computation
(89),
699--716
(2020)
\doiurl{10.1090/MCOM/3470}
\end{botherref}
\endbibitem

%%% 13
\bibitem[\protect\citeauthoryear{Hairer et~al.}{2006}]{hairer06gni}
\begin{bbook}
\bauthor{\bsnm{Hairer}, \binits{E.}},
\bauthor{\bsnm{Lubich}, \binits{C.}},
\bauthor{\bsnm{Wanner}, \binits{G.}}:
\bbtitle{Geometric Numerical Integration: Struc\-ture-pre\-ser\-ving Algorithms for Ordinary Differential Equations.},
\bedition{2nd} edn.
\bsertitle{Springer Series in Computational Mathematics},
vol. \bseriesno{31},
p. \bfpage{644}.
\bpublisher{Springer},
\blocation{Berlin}
(\byear{2006})
\end{bbook}
\endbibitem

%%% 14
\bibitem[\protect\citeauthoryear{Iserles et~al.}{2000}]{iserles_munthe-kaas_norsett_zanna_2000}
\begin{barticle}
\bauthor{\bsnm{Iserles}, \binits{A.}},
\bauthor{\bsnm{Munthe-Kaas}, \binits{H.Z.}},
\bauthor{\bsnm{N{\o}rsett}, \binits{S.P.}},
\bauthor{\bsnm{Zanna}, \binits{A.}}:
\batitle{Lie-group methods}.
\bjtitle{Acta Numerica}
\bvolume{9},
\bfpage{215}--\blpage{365}
(\byear{2000})
\doiurl{10.1017/S0962492900002154}
\end{barticle}
\endbibitem

%%% 15
\bibitem[\protect\citeauthoryear{Amsallem and Farhat}{2008}]{AmsallemFarhat2008}
\begin{barticle}
\bauthor{\bsnm{Amsallem}, \binits{D.}},
\bauthor{\bsnm{Farhat}, \binits{C.}}:
\batitle{Interpolation method for adapting reduced-order models and application to aeroelasticity}.
\bjtitle{AIAA Journal}
\bvolume{46}(\bissue{7}),
\bfpage{1803}--\blpage{1813}
(\byear{2008})
\doiurl{10.2514/1.35374}
\end{barticle}
\endbibitem

%%% 16
\bibitem[\protect\citeauthoryear{Zimmermann}{2014}]{Zimmermann2014}
\begin{barticle}
\bauthor{\bsnm{Zimmermann}, \binits{R.}}:
\batitle{A locally parametrized reduced order model for the linear frequency domain approach to time-accurate computational fluid dynamics}.
\bjtitle{SIAM Journal on Scientific Computing}
\bvolume{36}(\bissue{3}),
\bfpage{508}--\blpage{537}
(\byear{2014})
\end{barticle}
\endbibitem

%%% 17
\bibitem[\protect\citeauthoryear{Ciaramella et~al.}{2024}]{GentleGrassmannInt2024}
\begin{barticle}
\bauthor{\bsnm{Ciaramella}, \binits{G.}},
\bauthor{\bsnm{Gander}, \binits{M.J.}},
\bauthor{\bsnm{Vanzan}, \binits{T.}}:
\batitle{{A Gentle Introduction to Interpolation on the Grassmann Manifold}}.
\bjtitle{Oberwolfach Preprints}
(\bissue{OWP-2024-02})
(\byear{2024})
\doiurl{10.14760/OWP-2024-02-v2}
\end{barticle}
\endbibitem

%%% 18
\bibitem[\protect\citeauthoryear{Hess et~al.}{2023}]{HessQuainiRozza:2023}
\begin{botherref}
\oauthor{\bsnm{Hess}, \binits{M.}},
\oauthor{\bsnm{Quaini}, \binits{A.}},
\oauthor{\bsnm{Rozza}, \binits{G.}}:
A data-driven surrogate modeling approach for time-dependent incompressible navier-stokes equations with dynamic mode decomposition and manifold interpolation.
Advances in Computational Mathematics
\textbf{49}
(2023)
\doiurl{10.1007/s10444-023-10016-4}
\end{botherref}
\endbibitem

%%% 19
\bibitem[\protect\citeauthoryear{Zimmermann}{2021}]{Zim21}
\begin{bchapter}
\bauthor{\bsnm{Zimmermann}, \binits{R.}}:
\bctitle{Manifold interpolation}.
In: \beditor{\bsnm{Benner}, \binits{P.}},
\beditor{\bsnm{Grivet-Talocia}, \binits{S.}},
\beditor{\bsnm{Quarteroni}, \binits{A.}},
\beditor{\bsnm{Rozza}, \binits{G.}},
\beditor{\bsnm{Schilders}, \binits{W.}},
\beditor{\bsnm{Silveira}, \binits{L.M.}} (eds.)
\bbtitle{Volume 1 System- and Data-Driven Methods and Algorithms},
pp. \bfpage{229}--\blpage{274}.
\bpublisher{De Gruyter}, \blocation{???}
(\byear{2021}).
\doiurl{10.1515/9783110498967-007}
\end{bchapter}
\endbibitem

%%% 20
\bibitem[\protect\citeauthoryear{{El Omari} et~al.}{2025}]{ElOmari:2025}
\begin{barticle}
\bauthor{\bsnm{{El Omari}}, \binits{A.}},
\bauthor{\bsnm{{El Khlifi}}, \binits{M.}},
\bauthor{\bsnm{Cordier}, \binits{L.}}:
\batitle{Stiefel manifold interpolation for non-intrusive model reduction of parameterized fluid flow problems}.
\bjtitle{Journal of Computational Physics}
\bvolume{521},
\bfpage{113564}
(\byear{2025})
\doiurl{10.1016/j.jcp.2024.113564}
\end{barticle}
\endbibitem

%%% 21
\bibitem[\protect\citeauthoryear{Pfaller et~al.}{2020}]{Pfaller:2020}
\begin{barticle}
\bauthor{\bsnm{Pfaller}, \binits{M.R.}},
\bauthor{\bsnm{Cruz~Varona}, \binits{M.}},
\bauthor{\bsnm{Lang}, \binits{J.}},
\bauthor{\bsnm{Bertoglio}, \binits{C.}},
\bauthor{\bsnm{Wall}, \binits{W.A.}}:
\batitle{Using parametric model order reduction for inverse analysis of large nonlinear cardiac simulations}.
\bjtitle{International Journal for Numerical Methods in Biomedical Engineering}
\bvolume{36}(\bissue{4}),
\bfpage{3320}
(\byear{2020})
\doiurl{10.1002/cnm.3320}
\end{barticle}
\endbibitem

%%% 22
\bibitem[\protect\citeauthoryear{Helmke and Moore}{1994}]{HelmkeMoore1994}
\begin{bbook}
\bauthor{\bsnm{Helmke}, \binits{U.}},
\bauthor{\bsnm{Moore}, \binits{J.B.}}:
\bbtitle{Optimization and Dynamical Systems}.
\bsertitle{Communications \& Control Engineering}.
\bpublisher{Springer},
\blocation{London}
(\byear{1994})
\end{bbook}
\endbibitem

%%% 23
\bibitem[\protect\citeauthoryear{Zimmermann}{2020}]{ZimmermannHermite:2020}
\begin{barticle}
\bauthor{\bsnm{Zimmermann}, \binits{R.}}:
\batitle{{H}ermite interpolation and data processing errors on {R}iemannian matrix manifolds}.
\bjtitle{SIAM Journal on Scientific Computing}
\bvolume{42}(\bissue{5}),
\bfpage{2593}--\blpage{2619}
(\byear{2020})
\doiurl{10.1137/19M1282878}
\end{barticle}
\endbibitem

%%% 24
\bibitem[\protect\citeauthoryear{Jakubiak et~al.}{2006}]{Jakubiak:2006}
\begin{barticle}
\bauthor{\bsnm{Jakubiak}, \binits{J.}},
\bauthor{\bsnm{Leite}, \binits{F.S.}},
\bauthor{\bsnm{Rodrigues}, \binits{R.}}:
\batitle{A two-step algorithm of smooth spline generation on {R}iemannian manifolds}.
\bjtitle{Journal of Computational and Applied Mathematics}
\bvolume{194},
\bfpage{177}--\blpage{191}
(\byear{2006})
\doiurl{10.1016/j.cam.2005.07.003}
\end{barticle}
\endbibitem

%%% 25
\bibitem[\protect\citeauthoryear{S{\'e}guin and Kressner}{2024}]{SeguinKressner:2024}
\begin{barticle}
\bauthor{\bsnm{S{\'e}guin}, \binits{A.}},
\bauthor{\bsnm{Kressner}, \binits{D.}}:
\batitle{Hermite interpolation with retractions on manifolds}.
\bjtitle{BIT Numerical Mathematics}
\bvolume{64}(\bissue{4}),
\bfpage{42}
(\byear{2024})
\doiurl{10.1007/s10543-024-01023-y}
\end{barticle}
\endbibitem

%%% 26
\bibitem[\protect\citeauthoryear{Bergmann and Gousenbourger}{2018}]{BergmannGousenbourger:2018}
\begin{barticle}
\bauthor{\bsnm{Bergmann}, \binits{R.}},
\bauthor{\bsnm{Gousenbourger}, \binits{P.-Y.}}:
\batitle{A variational model for data fitting on manifolds by minimizing the acceleration of a {B}{\'e}zier curve}.
\bjtitle{Frontiers in Applied Mathematics and Statistics}
(\byear{2018})
\doiurl{10.3389/fams.2018.00059}
\end{barticle}
\endbibitem

%%% 27
\bibitem[\protect\citeauthoryear{Gousenbourger et~al.}{2018}]{GouseMassartAbsil:2018}
\begin{barticle}
\bauthor{\bsnm{Gousenbourger}, \binits{P.-Y.}},
\bauthor{\bsnm{Massart}, \binits{E.}},
\bauthor{\bsnm{Absil}, \binits{P.-A.}}:
\batitle{Data fitting on manifolds with composite {B}{\'e}zier-like curves and blended cubic splines}.
\bjtitle{Journal of Mathematical Imaging and Vision}
\bvolume{online},
\bfpage{1}--\blpage{27}
(\byear{2018})
\doiurl{10.1007/s10851-018-0865-2}
\end{barticle}
\endbibitem

%%% 28
\bibitem[\protect\citeauthoryear{Nava-Yazdani and Polthier}{2013}]{Polthier:2013}
\begin{barticle}
\bauthor{\bsnm{Nava-Yazdani}, \binits{E.}},
\bauthor{\bsnm{Polthier}, \binits{K.}}:
\batitle{{D}e {C}asteljau's algorithm on manifolds}.
\bjtitle{Computer Aided Geometric Design}
\bvolume{30}(\bissue{7}),
\bfpage{722}--\blpage{732}
(\byear{2013})
\doiurl{10.1016/j.cagd.2013.06.002}
\end{barticle}
\endbibitem

%%% 29
\bibitem[\protect\citeauthoryear{Popiel and Noakes}{2007}]{Noakes:2007}
\begin{barticle}
\bauthor{\bsnm{Popiel}, \binits{T.}},
\bauthor{\bsnm{Noakes}, \binits{L.}}:
\batitle{B{\'e}zier curves and {C2} interpolation in {R}iemannian manifolds}.
\bjtitle{Journal of Approximation Theory}
\bvolume{148}(\bissue{2}),
\bfpage{111}--\blpage{127}
(\byear{2007})
\end{barticle}
\endbibitem

%%% 30
\bibitem[\protect\citeauthoryear{Samir and Adouani}{2019}]{SAMIR:2019}
\begin{barticle}
\bauthor{\bsnm{Samir}, \binits{C.}},
\bauthor{\bsnm{Adouani}, \binits{I.}}:
\batitle{{C}1 interpolating {B}{\'e}zier path on {R}iemannian manifolds, with applications to {3D} shape space}.
\bjtitle{Applied Mathematics and Computation}
\bvolume{348},
\bfpage{371}--\blpage{384}
(\byear{2019})
\doiurl{10.1016/j.amc.2018.11.060}
\end{barticle}
\endbibitem

%%% 31
\bibitem[\protect\citeauthoryear{Dyn and Sharon}{2017}]{Dyn:2017}
\begin{barticle}
\bauthor{\bsnm{Dyn}, \binits{N.}},
\bauthor{\bsnm{Sharon}, \binits{N.}}:
\batitle{Manifold-valued subdivision schemes based on geodesic inductive averaging}.
\bjtitle{Journal of Computational and Applied Mathematics}
\bvolume{311},
\bfpage{54}--\blpage{67}
(\byear{2017})
\doiurl{10.1016/j.cam.2016.07.008}
\end{barticle}
\endbibitem

%%% 32
\bibitem[\protect\citeauthoryear{Wallner and Dyn}{2005}]{Wallner:2005}
\begin{barticle}
\bauthor{\bsnm{Wallner}, \binits{J.}},
\bauthor{\bsnm{Dyn}, \binits{N.}}:
\batitle{Convergence and {C1} analysis of subdivision schemes on manifolds by proximity}.
\bjtitle{Computer Aided Geometric Design}
\bvolume{22}(\bissue{7}),
\bfpage{593}--\blpage{622}
(\byear{2005})
\doiurl{10.1016/j.cagd.2005.06.003}
\end{barticle}
\endbibitem

%%% 33
\bibitem[\protect\citeauthoryear{Hofit et~al.}{2023}]{BenZion:2022}
\begin{barticle}
\bauthor{\bsnm{Hofit}, \binits{B.-Z.V.}},
\bauthor{\bsnm{Nira}, \binits{D.}},
\bauthor{\bsnm{Nir}, \binits{S.}}:
\batitle{Geometric hermite interpolation in {$\mathbb{R}^n$} by refinements}.
\bjtitle{Advances in Computational Mathematics}
\bvolume{49}(\bissue{3}),
\bfpage{38}
(\byear{2023})
\doiurl{10.1007/s10444-023-10037-z}
\end{barticle}
\endbibitem

%%% 34
\bibitem[\protect\citeauthoryear{Attali et~al.}{2007}]{Attali2007}
\begin{bchapter}
\bauthor{\bsnm{Attali}, \binits{D.}},
\bauthor{\bsnm{Edelsbrunner}, \binits{H.}},
\bauthor{\bsnm{Mileyko}, \binits{Y.}}:
\bctitle{Weak witnesses for {D}elaunay triangulations of submanifolds}.
In: \bbtitle{Proceedings of the 2007 ACM Symposium on Solid and Physical Modeling}.
\bsertitle{SPM '07},
pp. \bfpage{143}--\blpage{150}.
\bpublisher{Association for Computing Machinery},
\blocation{New York, NY, USA}
(\byear{2007}).
\doiurl{10.1145/1236246.1236267}
\end{bchapter}
\endbibitem

%%% 35
\bibitem[\protect\citeauthoryear{Mataigne et~al.}{2024}]{mataigne2024bounds}
\begin{botherref}
\oauthor{\bsnm{Mataigne}, \binits{S.}},
\oauthor{\bsnm{Absil}, \binits{P.-A.}},
\oauthor{\bsnm{Miolane}, \binits{N.}}:
Bounds on the geodesic distances on the {S}tiefel manifold for a family of {R}iemannian metrics
(2024).
\url{https://arxiv.org/abs/2408.07072}
\end{botherref}
\endbibitem

%%% 36
\bibitem[\protect\citeauthoryear{Jacobsson et~al.}{2024}]{Jacobsson:2024}
\begin{botherref}
\oauthor{\bsnm{Jacobsson}, \binits{S.}},
\oauthor{\bsnm{Vandebril}, \binits{R.}},
\oauthor{\bsnm{Veken}, \binits{J.}},
\oauthor{\bsnm{Vannieuwenhoven}, \binits{N.}}:
Approximating maps into manifolds with lower curvature bounds
(2024)
\end{botherref}
\endbibitem

%%% 37
\bibitem[\protect\citeauthoryear{Gallier}{2011}]{gallier2011geometric}
\begin{bbook}
\bauthor{\bsnm{Gallier}, \binits{J.}}:
\bbtitle{Geometric Methods and Applications: For Computer Science and Engineering}.
\bsertitle{Texts in Applied Mathematics}.
\bpublisher{Springer}, \blocation{???}
(\byear{2011}).
\doiurl{10.1007/978-1-4613-0137-0}
\end{bbook}
\endbibitem

%%% 38
\bibitem[\protect\citeauthoryear{Bendokat et~al.}{2024}]{Bendokat:2024}
\begin{barticle}
\bauthor{\bsnm{Bendokat}, \binits{T.}},
\bauthor{\bsnm{Zimmermann}, \binits{R.}},
\bauthor{\bsnm{Absil}, \binits{P.-A.}}:
\batitle{A {G}rassmann manifold handbook: Basic geometry and computational aspects}.
\bjtitle{Advances in Computational Mathematics}
\bvolume{50}(\bissue{1}),
\bfpage{6}
(\byear{2024})
\doiurl{10.1007/s10444-023-10090-8}
\end{barticle}
\endbibitem

%%% 39
\bibitem[\protect\citeauthoryear{Saad}{1992}]{Eigenvalue_Problems_Saad}
\begin{bbook}
\bauthor{\bsnm{Saad}, \binits{Y.}}:
\bbtitle{Numerical Methods for Large Eigenvalue Problems}.
\bsertitle{Algortihms and Architectures for Advanced Scientific Computing}.
\bpublisher{{M}anchester {U}niversity {P}ress},
\blocation{Manchester, {UK}}
(\byear{1992})
\end{bbook}
\endbibitem

%%% 40
\bibitem[\protect\citeauthoryear{Lee}{2012}]{LeeSmooth:2012}
\begin{bbook}
\bauthor{\bsnm{Lee}, \binits{J.M.}}:
\bbtitle{{I}ntroduction to {S}mooth {M}anifolds}.
\bsertitle{Graduate Texts in Mathematics}.
\bpublisher{Springer}, \blocation{???}
(\byear{2012})
\end{bbook}
\endbibitem

%%% 41
\bibitem[\protect\citeauthoryear{Wong}{1967}]{Wong1967}
\begin{barticle}
\bauthor{\bsnm{Wong}, \binits{Y.-C.}}:
\batitle{Differential geometry of {G}rassmann manifolds}.
\bjtitle{Proceedings of the National Academy of Sciences of the United States of America}
\bvolume{57},
\bfpage{589}--\blpage{594}
(\byear{1967})
\end{barticle}
\endbibitem

%%% 42
\bibitem[\protect\citeauthoryear{Wong}{1968}]{Wong1968}
\begin{barticle}
\bauthor{\bsnm{Wong}, \binits{Y.-C.}}:
\batitle{Sectional curvatures of {G}rassmann manifolds}.
\bjtitle{Proceedings of the National Academy of Sciences of the United States of America}
\bvolume{60}(\bissue{1}),
\bfpage{75}--\blpage{79}
(\byear{1968})
\end{barticle}
\endbibitem

%%% 43
\bibitem[\protect\citeauthoryear{K{\"u}hnel}{2015}]{kuhnel:2015}
\begin{bbook}
\bauthor{\bsnm{K{\"u}hnel}, \binits{W.}}:
\bbtitle{Differential Geometry: Curves -- Surfaces -- Manifolds},
\bedition{3}rd edn.
\bsertitle{Student Mathematical Library}.
\bpublisher{American Mathematical Society}, \blocation{???}
(\byear{2015})
\end{bbook}
\endbibitem

%%% 44
\bibitem[\protect\citeauthoryear{Amsallem}{2010}]{Amsallem2010}
\begin{botherref}
\oauthor{\bsnm{Amsallem}, \binits{D.}}:
Interpolation on manifolds of {CFD}-based fluid and finite element-based structural reduced-order models for on-line aeroelastic prediction.
PhD thesis,
Stanford University
(2010)
\end{botherref}
\endbibitem

%%% 45
\bibitem[\protect\citeauthoryear{Klingenberg}{1978}]{Klingenberg:1978}
\begin{bbook}
\bauthor{\bsnm{Klingenberg}, \binits{W.}}:
\bbtitle{A Course in Differential Geometry}.
\bsertitle{Graduate Texts in Mathematics},
vol. \bseriesno{51}.
\bpublisher{Springer},
\blocation{New York}
(\byear{1978})
\end{bbook}
\endbibitem

%%% 46
\bibitem[\protect\citeauthoryear{Quarteroni et~al.}{2007}]{Alfio:2007}
\begin{bbook}
\bauthor{\bsnm{Quarteroni}, \binits{A.}},
\bauthor{\bsnm{Sacco}, \binits{R.}},
\bauthor{\bsnm{Saleri}, \binits{F.}}:
\bbtitle{Numerical Mathematics}.
\bpublisher{Springer}, \blocation{???}
(\byear{2007})
\end{bbook}
\endbibitem

%%% 47
\bibitem[\protect\citeauthoryear{Goreinov et~al.}{1997}]{Goreinov:1997}
\begin{barticle}
\bauthor{\bsnm{Goreinov}, \binits{S.A.}},
\bauthor{\bsnm{Tyrtyshnikov}, \binits{E.E.}},
\bauthor{\bsnm{Zamarashkin}, \binits{N.L.}}:
\batitle{A theory of pseudoskeleton approximations}.
\bjtitle{Linear Algebra and its Applications}
\bvolume{261}(\bissue{1}),
\bfpage{1}--\blpage{21}
(\byear{1997})
\doiurl{10.1016/S0024-3795(96)00301-1}
\end{barticle}
\endbibitem

%%% 48
\bibitem[\protect\citeauthoryear{Chaturantabut and Sorensen}{2010}]{Chaturantabut:2010}
\begin{barticle}
\bauthor{\bsnm{Chaturantabut}, \binits{S.}},
\bauthor{\bsnm{Sorensen}, \binits{D.C.}}:
\batitle{Nonlinear model reduction via discrete empirical interpolation}.
\bjtitle{SIAM Journal on Scientific Computing}
\bvolume{32}(\bissue{5}),
\bfpage{2737}--\blpage{2764}
(\byear{2010})
\doiurl{10.1137/090766498}
\end{barticle}
\endbibitem

%%% 49
\bibitem[\protect\citeauthoryear{Drma\v{c} and Gugercin}{2016}]{DrmacGugercin:2015}
\begin{barticle}
\bauthor{\bsnm{Drma\v{c}}, \binits{Z.}},
\bauthor{\bsnm{Gugercin}, \binits{S.}}:
\batitle{A new selection operator for the discrete empirical interpolation method---improved a priori error bound and extensions}.
\bjtitle{SIAM Journal on Scientific Computing}
\bvolume{38}(\bissue{2}),
\bfpage{631}--\blpage{648}
(\byear{2016})
\doiurl{10.1137/15M1019271}
\end{barticle}
\endbibitem

%%% 50
\bibitem[\protect\citeauthoryear{Chandrasekaran and Ipsen}{1994}]{ChandrasekaranIpsen:1994}
\begin{barticle}
\bauthor{\bsnm{Chandrasekaran}, \binits{S.}},
\bauthor{\bsnm{Ipsen}, \binits{I.C.F.}}:
\batitle{On rank-revealing factorisations}.
\bjtitle{SIAM Journal on Matrix Analysis and Applications}
\bvolume{15}(\bissue{2}),
\bfpage{592}--\blpage{622}
(\byear{1994})
\doiurl{10.1137/S0895479891223781}
\end{barticle}
\endbibitem

%%% 51
\bibitem[\protect\citeauthoryear{Gu and Eisenstat}{1996}]{GuEisenstat:1996}
\begin{barticle}
\bauthor{\bsnm{Gu}, \binits{M.}},
\bauthor{\bsnm{Eisenstat}, \binits{S.C.}}:
\batitle{Efficient algorithms for computing a strong rank-revealing qr factorization}.
\bjtitle{SIAM Journal on Scientific Computing}
\bvolume{17}(\bissue{4}),
\bfpage{848}--\blpage{869}
(\byear{1996})
\doiurl{10.1137/0917055}
\end{barticle}
\endbibitem

%%% 52
\bibitem[\protect\citeauthoryear{Damle et~al.}{2024}]{damle:2024}
\begin{botherref}
\oauthor{\bsnm{Damle}, \binits{A.}},
\oauthor{\bsnm{Glas}, \binits{S.}},
\oauthor{\bsnm{Townsend}, \binits{A.}},
\oauthor{\bsnm{Yu}, \binits{A.}}:
How to reveal the rank of a matrix?
(2024).
\url{https://arxiv.org/abs/2405.04330}
\end{botherref}
\endbibitem

%%% 53
\bibitem[\protect\citeauthoryear{Goreinov et~al.}{2010}]{Goreinov:2010}
\begin{bchapter}
\bauthor{\bsnm{Goreinov}, \binits{S.A.}},
\bauthor{\bsnm{Oseledets}, \binits{I.V.}},
\bauthor{\bsnm{Savostyanov}, \binits{D.V.}},
\bauthor{\bsnm{Tyrtyshnikov}, \binits{E.E.}},
\bauthor{\bsnm{Zamarashkin}, \binits{N.L.}}:
\bctitle{How to find a good submatrix}.
In: \bbtitle{Matrix Methods: Theory, Algorithms And Applications: Dedicated to the Memory of Gene Golub},
pp. \bfpage{247}--\blpage{256}.
\bpublisher{World Scientific}, \blocation{???}
(\byear{2010}).
\doiurl{10.1142/9789812836021_0015}
\end{bchapter}
\endbibitem

%%% 54
\bibitem[\protect\citeauthoryear{Higham}{2008}]{Higham:2008:FM}
\begin{bbook}
\bauthor{\bsnm{Higham}, \binits{N.J.}}:
\bbtitle{Functions of Matrices: {Theory} and Computation},
p. \bfpage{425}.
\bpublisher{Society for Industrial and Applied Mathematics},
\blocation{Philadelphia, PA, USA}
(\byear{2008})
\end{bbook}
\endbibitem

%%% 55
\bibitem[\protect\citeauthoryear{Zimmermann and Stoye}{2025}]{zimmermannstoyeSIMAX:2024}
\begin{barticle}
\bauthor{\bsnm{Zimmermann}, \binits{R.}},
\bauthor{\bsnm{Stoye}, \binits{J.}}:
\batitle{{High Curvature Means Low Rank: On the Sectional Curvature of Grassmann and Stiefel Manifolds and the Underlying Matrix Trace Inequalities}}.
\bjtitle{SIAM Journal on Matrix Analysis and Applications}
\bvolume{46}(\bissue{1}),
\bfpage{748}--\blpage{779}
(\byear{2025})
\doiurl{10.1137/24M1655755}
{\href{https://arxiv.org/abs/https://doi.org/10.1137/24M1655755}{{https://doi.org/10.1137/24M1655755}}}
\end{barticle}
\endbibitem

%%% 56
\bibitem[\protect\citeauthoryear{Hall and Newell}{1979}]{Hall:1979}
\begin{barticle}
\bauthor{\bsnm{Hall}, \binits{W.S.}},
\bauthor{\bsnm{Newell}, \binits{M.L.}}:
\batitle{The mean value theorem for vector valued functions: A simple proof}.
\bjtitle{Mathematics Magazine}
\bvolume{52}(\bissue{3}),
\bfpage{157}--\blpage{158}
(\byear{1979})
\doiurl{10.2307/2690276}
\end{barticle}
\endbibitem

%%% 57
\bibitem[\protect\citeauthoryear{Walter et~al.}{2012}]{Walter:2012}
\begin{barticle}
\bauthor{\bsnm{Walter}, \binits{S.F.}},
\bauthor{\bsnm{Lehmann}, \binits{L.}},
\bauthor{\bsnm{and}, \binits{R.L.}}:
\batitle{On evaluating higher-order derivatives of the {QR} decomposition of tall matrices with full column rank in forward and reverse mode algorithmic differentiation}.
\bjtitle{Optimization Methods and Software}
\bvolume{27}(\bissue{2}),
\bfpage{391}--\blpage{403}
(\byear{2012})
\doiurl{10.1080/10556788.2011.610454}
\end{barticle}
\endbibitem

%%% 58
\bibitem[\protect\citeauthoryear{A.~Mallot}{2024}]{Mallot:2024}
\begin{bbook}
\bauthor{\bsnm{A.~Mallot}, \binits{H.}}:
\bbtitle{{C}omputational {N}euroscience}.
\bpublisher{Springer}, \blocation{???}
(\byear{2024})
\end{bbook}
\endbibitem

%%% 59
\bibitem[\protect\citeauthoryear{Cebrián-Lacasa et~al.}{2024}]{Lacasa:2024}
\begin{barticle}
\bauthor{\bsnm{Cebrián-Lacasa}, \binits{D.}},
\bauthor{\bsnm{Parra-Rivas}, \binits{P.}},
\bauthor{\bsnm{Ruiz-Reynés}, \binits{D.}},
\bauthor{\bsnm{Gelens}, \binits{L.}}:
\batitle{Six decades of the {F}itzhugh–{N}agumo model: A guide through its spatio-temporal dynamics and influence across disciplines}.
\bjtitle{Physics Reports}
\bvolume{1096},
\bfpage{1}--\blpage{39}
(\byear{2024})
\doiurl{10.1016/j.physrep.2024.09.014}
\end{barticle}
\endbibitem

%%% 60
\bibitem[\protect\citeauthoryear{Chaturantabut}{2011}]{Chaturantabut:2011}
\begin{botherref}
\oauthor{\bsnm{Chaturantabut}, \binits{S.}}:
Nonlinear model reduction via discrete empirical interpolation.
PhD thesis,
Rice Unversity
(2011)
\end{botherref}
\endbibitem

%%% 61
\bibitem[\protect\citeauthoryear{do~Carmo}{1992}]{DoCarmo2013riemannian}
\begin{bbook}
\bauthor{\bsnm{Carmo}, \binits{M.P.}}:
\bbtitle{{R}iemannian Geometry}.
\bsertitle{Mathematics: Theory \& Applications}.
\bpublisher{Birkh{\"a}user Boston}, \blocation{???}
(\byear{1992})
\end{bbook}
\endbibitem

%%% 62
\bibitem[\protect\citeauthoryear{Lee}{1997}]{Lee1997riemannian}
\begin{bbook}
\bauthor{\bsnm{Lee}, \binits{J.M.}}:
\bbtitle{{Riemannian Manifolds: an Introduction to Curvature}}.
\bpublisher{Springer},
\blocation{New York -- Berlin -- Heidelberg}
(\byear{1997})
\end{bbook}
\endbibitem

\end{thebibliography}

\end{document}